\documentclass{amsart}
\usepackage{verbatim}
\usepackage{hyperref}
\usepackage{amsmath,amssymb,graphicx,float,epsf}
\usepackage{graphicx}
\usepackage{amsmath}
\usepackage{amsthm}
\usepackage{amsfonts}

\newtheorem{theorem}{Theorem}[section]
\newtheorem{corollary}[theorem]{Corollary}
\newtheorem{lemma}[theorem]{Lemma}
\newtheorem{definition}[theorem]{Definition}
\allowdisplaybreaks

\begin{document}
\author[Korobenko]{Lyudmila Korobenko}
\address{University of Calgary\\
Calgary, Alberta\\
lkoroben@ucalgary.ca}
\author[Rios]{Cristian Rios}
\address{University of Calgary\\
Calgary, Alberta\\
crios@ucalgary.ca}
\title[Hypoellipticity of Infinitely Degenerate Operators]{Hypoellipticity
of a Class of Infinitely Degenerate Second Order Operators and Systems}
\thanks{The second author is supported by the Natural Sciences and
Engineering Research Council of Canada.}

\begin{abstract}
In this paper we establish a hypoellipticity result for second order linear
operators comprised by a linear combination, with infinite vanishing
coefficients, of subelliptic operators in separate spaces. This generalizes
previous known results.
\end{abstract}

\maketitle

\section{Introduction}

An operator $L$ acting on $\mathcal{D}^{\prime }\left( \mathbb{R}^{n}\right) 
$, the set of distributions, is said to be hypoelliptic if whenever $u\in 
\mathcal{D}^{\prime }\left( \mathbb{R}^{n}\right) $ and $Lu\in C^{\infty
}\left( \mathbb{R}^{n}\right) $ then $u\in C^{\infty }\left( \mathbb{R}%
^{n}\right) $. A sufficient condition for an operator to be hypoelliptic is
subellipticity: $L$ is subelliptic if there exists some $\varepsilon ,C>0$
such \ that 
\begin{equation}
||u||_{\varepsilon }^{2}\leq C\left( \left\vert \left( Lu,u\right)
\right\vert +\left\Vert u\right\Vert ^{2}\right) \qquad \text{for all }u\in
C_{0}^{\infty }\left( \mathbb{R}^{n}\right) ;  \label{subel}
\end{equation}%
$\left\Vert \cdot \right\Vert _{s}$ denotes the Sobolev norm of order $s\in 
\mathbb{R}$ (see Definition \ref{def-Sobolev-norm} below), and $\left\Vert
\cdot \right\Vert =\left\Vert \cdot \right\Vert _{0}$ is the $L^{2}$ norm in 
$\mathbb{R}^{n}$. Some necessary and sufficient conditions for
subellipticity have been established in terms of associated vector fields by
H\"{o}rmander in his pivotal paper \cite{Hor}; and in terms of subunit
metric balls by Fefferman and Phong \cite{Fef}. Subelliptic operators may
have ellipticity vanishing locally to at most a finite order.

An operator with infinitely vanishing\ ellipticity is not subelliptic, such
operators do not satisfy the H\"{o}rmander condition. The first known
hypoellipticity results for infinitely degenerate operators are due to Fedi%
\u{\i} \cite{Fedii}, where the simplest example is $P=\partial
_{x}^{2}+k\left( x\right) \partial _{y}^{2}$ with $k\left( x\right) >0$ for $%
x\neq 0$, $\sqrt{k}$ is smooth and it is allowed to vanish to any order at \
the origin. A different criterion for hypoellipticity was developed by
Morimoto in Section 2 of \cite{Mor}, where he generalizes the seminal
techniques from \cite{Fedii}. Other sufficient conditions for
hypoellipticity where obtained by the same author in \cite{Mor1}, where the
left hand side on the subellipticity condition (\ref{subel}) is replaced by
logarithmic Sobolev norms.

The hypoellipticity of semilinear operators with principal part satisfying
the H\"{o}rmander condition was established in \cite{Tri}. Certain
quasilinear operators with infinitely vanishing ellipticity have been
studied in two dimensions by Sawyer and Wheeden motivated by applications to
Monge-Amp\`{e}re equations \cite{Saw0, Saw}; Rios et al extended these
results to a wider class of infinitely vanishing quasilinear equations in
higher dimensions \cite{Rios2, Rios}. However, hypoellipticity was only
obtained for continuous solutions. Previous nonlinear hypoellipticity
results had also required extra hypothesis on solutions: in \cite{Xu}
quasilinear subelliptic systems are considered, and hypoellipticity is
obtained for continuous solutions; in \cite{Mor3, Mor4} hypoellipticity is
obtained for bounded solutions of certain infinitely degenerate quasilinear
equations.

Returning to the linear case, Kusuoka and Stroock extended Fedi\u{\i}'s two
dimensional result to the case when only $k$ is required to be smooth and it
may vanish at any order at the origin \cite{Kus}. However, in \cite{Kus} the
authors also showed that in higher dimensions hypoellipticity may fail for
certain linear operators depending on the vanishing ellipticity order; they
in fact obtained a quite spectacular characterization of hypoellipticity for 
$Q=\partial _{x}^{2}+k\left( x\right) \partial _{y}^{2}+\partial _{z}^{2}$: $%
Q$ is hypoelliptic if and only if $\lim_{x\rightarrow 0}x\log k\left(
x\right) =0$. Their proofs rely on the Malliavin Calculus. In \cite{Mor},
Morimoto, using non-probabilistic methods, extended Kusuoka and Strook's
result to pseudodifferential operators of the form $R=a\left(
x,y,D_{x}\right) +g\left( x^{\prime }\right) b\left( x,y,D_{y}\right) $ in $%
\mathbb{R}^{n}=\mathbb{R}_{x}^{n_{1}}\times \mathbb{R}_{y}^{n_{2}}$ , where $%
a$ and $b$ are strongly elliptic pseudodifferential operators, $x=\left(
x^{\prime },x^{\prime \prime }\right) \in \mathbb{R}^{n_{1}}=\mathbb{R}%
^{d_{1}}\times \mathbb{R}^{d_{2}}$, $g$ is smooth, $g\left( x^{\prime
}\right) >0$ for $x^{\prime }\neq 0$ and $\lim_{x^{\prime }\rightarrow
0}\left\vert x^{\prime }\right\vert \left\vert \log g\left( x^{\prime
}\right) \right\vert =0$.

In fact, Fedi\u{\i}'s two dimensional result does extend to operators in
higher dimensions regardless of the order of vanishing if their structure is
similar that of the two dimensional operator $P$. Indeed, $P$ may be written
in the form $P=L_{1}+k\left( x\right) L_{2}$ in $\mathbb{R}\times \mathbb{R}$%
, where $L_{1}=\partial _{x}^{2}$ and $L_{2}=\partial _{y}^{2}$ are one
dimensional elliptic operators (notice that the coefficient $k$ does not
depend on the second variable). With this perspective, Morimoto generalized
Fedi\u{\i}'s result to pseudodifferential operators of the from $R=a\left(
x,y,D_{x}\right) +g\left( x\right) b\left( x,y,D_{y}\right) $ in $\mathbb{R}%
^{n}=\mathbb{R}_{x}^{n_{1}}\times \mathbb{R}_{y}^{n_{2}}$ , where $a$ and $b$
are strongly elliptic pseudodifferential operators, $g\left( x\right) >0$
for $x\neq 0$, $g$ is smooth and it can vanish at any order at the origin 
\cite{Mor0}. Over a decade later Kohn \cite{Kohn} proved the hypoellipticity
of $R$ in the case that $a\left( x,D_{x}\right) =L_{1}$ and $b\left(
y,D_{y}\right) =L_{2}$ are only assumed to be differential operators 
\begin{equation}
L_{k}=-\sum\limits_{i,j=1}^{n_{k}}a_{ij}^{k}\left( x^{k}\right) \frac{%
\partial ^{2}}{\partial x_{i}^{k}\partial x_{j}^{k}}+\sum%
\limits_{i=1}^{n_{k}}b_{i}^{k}\left( x^{k}\right) \frac{\partial }{\partial
x_{i}^{k}}+c^{k}\left( x^{k}\right) ,  \label{fam}
\end{equation}%
which are subelliptic in $\mathbb{R}^{n_{k}}$, $k=1,2$, respectively.

The purpose of this paper is to generalize Kohn's result to an arbitrary
finite number of subelliptic operators in separate variables, extending the
Fedi\u{\i}'s type structure modeled in \cite{Mor0, Kohn}. We also obtain
hypoellipticity for systems of linear operators with a similar infinite
degeneracy.

\begin{definition}[Subelliptic operator]
Let $L$ be an operator defined by 
\begin{equation}
L=-\sum\limits_{1}^{n}a_{ij}(x)\frac{\partial ^{2}}{\partial x_{i}\partial
x_{j}}+\sum\limits_{1}^{n}b_{i}(x)\frac{\partial }{\partial x_{i}}+c(x)
\label{subop}
\end{equation}%
where $a_{ij},\;b_{i},\;c\in C^{\infty }(U)$ and $\left( a_{ij}\right)
_{i,j=1}^{n}\geq 0$. Then $L$ is \emph{subelliptic at $x_{0}\in \mathbb{R}%
^{n}$} if there exists a neighborhood $U$ of $x_{0}$ and positive constants $%
\varepsilon $ and $C$ such that (\ref{subel}) holds for all $u\in
C_{0}^{\infty }(\mathbb{R}^{n})$. $L$ is called \emph{subelliptic} if it is
subelliptic at each point of $\mathbb{R}^{n}$.
\end{definition}

\begin{definition}[Hypoellipticity without loss of derivatives]
A linear operator $L$ acting on distributions in $\mathbb{R}^{n}$ is \emph{%
hypoelliptic} if and only if whenever $Lu\in C^{\infty }\left( \mathbb{R}%
^{n}\right) $ for some distribution $u$, then $u\in C^{\infty }\left( 
\mathbb{R}^{n}\right) $. \newline
$L$ is said to be \emph{hypoelliptic without loss of derivatives} if for
given any open set $U\subset \mathbb{R}^{n}$, then if $\zeta Lu\in
H^{s}\left( \mathbb{R}^{n}\right) $ for all $\zeta \in C_{0}^{\infty }\left(
U\right) $ then $\zeta u\in H^{s}\left( \mathbb{R}^{n}\right) $ for all $%
\zeta \in C_{0}^{\infty }\left( U\right) $.
\end{definition}

For fixed positive integers $n_{k}$, $k=1,\cdots ,m$, we denote $x\in
\prod_{k=1}^{m}\mathbb{R}^{n_{k}}$ as $x=\left( x^{1},\dots
,x^{n_{m}}\right) \in \mathbb{R}^{n}$, with 
\begin{equation*}
x^{k}=\left( x_{1}^{k},\dots ,x_{n_{k}}^{k}\right) \in \mathbb{R}^{n_{k}}%
\text{,\quad }k=1,\cdots ,m\text{,\quad }n=\sum_{k=1}^{m}n_{k},
\end{equation*}%
and we let $\overline{x^{k}}$ be the vector obtained from $x$ by omitting $%
x^{k}$, i.e.%
\begin{equation*}
\overline{x^{k}}=\left( x^{1},\;\ldots \;,x^{k-1},x^{k+1},\;\ldots
\;x^{m}\right) .
\end{equation*}%
In the scalar case, our main result is the following:

\begin{theorem}
\label{main} Suppose that $L_{k}$ as in (\ref{fam}) are subelliptic,and $%
\lambda _{k}=\lambda _{k}(\overline{x^{k}})\geq 0$, $k=1,\cdots ,m$, are
smooth functions. Assume $\lambda _{1}\equiv 1$, and that for $2\leq k\leq m$%
, $\lambda _{k}\left( \overline{x^{k}}\right) >0$ for $\overline{x^{k}}\neq
0 $. Then the operator $L$ defined by 
\begin{equation}
L=\sum\limits_{k=1}^{m}\lambda _{k}L_{k}  \label{operator}
\end{equation}%
is hypoelliptic without loss of derivatives in $\mathbb{R}^{n}$.
\end{theorem}

The important cases of the above result are when some of the coefficients $%
\lambda _{k}$ have a zero of infinite order at the $n_{k}$-dimensional
subspaces $\overline{x^{k}}=0$ in $\mathbb{R}^{n}$. Because of the local
nature of the theorem, our results easily generalize to the case when $%
\sum_{k=1}^{m}\lambda _{k}>0$, and $\lambda _{k}$ has isolated zeroes in $%
\prod_{j\neq k}\mathbb{R}^{n_{j}}$, $k=1,\dots ,m$.

Note that in the case $m=2$ considered by Kohn \cite{Kohn} the coefficient $%
\lambda =\lambda _{2}$ was allowed to have zeroes of finite order outside $%
\overline{x_{2}}=0$. In this case, the operator $L\left( x^{1},x^{2}\right)
=L_{1}\left( x^{2}\right) +\lambda \left( x^{1}\right) L_{2}\left(
x^{2}\right) $ is subelliptic whenever $\lambda $ has a zero of finite order
and $L_{1}$, $L_{2}$ are subelliptic. However, when $m\geq 3$ this result is
not true. Indeed, the operators $L_{1}=-\partial _{x}^{2}-x^{2}\partial
_{y}^{2}-y^{2}\partial z^{2}$ and $L_{2}=-\partial _{x}^{2}-z^{2}\partial
_{y}^{2}-y^{2}\partial z^{2}$ are not subelliptic in $\mathbb{R}^{3}$ since
they are sum of the squares of analytic vector fields which do not satisfy
the H\"{o}rmander condition. Now, $L_{1}$ is hypoelliptic while $L_{2}$ it
is not. See Theorem 1 in \cite{Morio} to check the first assertion. The
proof in \cite{Morio} relies on the special structure of $L_{1}$, in which
the vanishing order of the coefficients is restricted. We consider a
different structure, where the degeneracy is localized in space but there is
no restrictions to the order of vanishing. On the other hand, to check that $%
L_{2}$ is not hypoelliptic, it is enough to note its action on the
distribution $u=\delta _{yz}$, where $\delta _{yz}$ is the Dirac delta
function at the origin in $\mathbb{R}^{2}$. Since $L_{2}$ is self-adjoint,
for any test function ${\varphi \in C}_{0}^{\infty }\left( \mathbb{R}%
^{3}\right) $ we have 
\begin{equation*}
\left\langle L_{2}u,{\varphi }\right\rangle =-\left\langle \delta _{yz},{%
\varphi }_{xx}+z^{2}{\varphi }_{yy}+y^{2}{\varphi }_{zz}\right\rangle
=-\int_{\mathbb{R}}{\varphi }_{xx}\left( x,0,0\right) ~dx=0.
\end{equation*}%
These examples illustrate one of the difficulties in generalizing Kohn's
result to the structure (\ref{operator}) including more than two summands.

Our hypoellipticity result extends to linear systems of equations. Our
interest in systems primarily arises from a study of an $n$-dimensional
Monge-Amp\`{e}re problem. Application of a partial Legendre transformation
leads to a system of quasilinear equations. Some results on the regularity
of solutions to the quasilinear system associated to an $n$-dimensional
Monge-Amp\`{e}re equation were obtained in \cite{Rios1}. In the present
paper we consider a general system of second order linear equations. We do
not assume any control on the vanishing of the operators' coefficients, so
in general vanishing can be infinite. Linear systems have been studied by
many authors and there is a more or less established elliptic theory \cite%
{Lad, Horm}. However, when ellipticity fails much less is known.

We now introduce some notation pertinent to dealing with systems of
equations. We let $\mathbf{u}\left( x\right) =\left( u_{1}\left( x\right)
,\dots ,u_{N}\left( x\right) \right) ^{t}$ be a (column) vector function in $%
\mathbb{R}^{n}$. Given the grouped variables $x^{k}=\left( x_{1}^{k},\dots
,x_{n_{k}}^{k}\right) \in \mathbb{R}^{n_{k}}$ as before, $1\leq k\leq m$, we
denote by $\nabla _{k}\mathbf{u}$ the $N\cdot n_{k}$ column vector%
\begin{equation*}
\nabla _{k}\mathbf{u}=\left( \nabla _{k}u_{p}\right) _{p=1}^{N}\in \mathbb{R}%
\left( N\otimes n_{k}\right) .
\end{equation*}%
To make clear the structure of such vectors, we say that $\nabla _{k}\mathbf{%
u}\in \mathbb{R}\left( N\otimes n_{k}\right) $ Let $\mathbf{A}^{k}$ be an $%
N\times N$ matrix with $n_{k}\times n_{k}$ matrices as its elements, we
write $\mathbf{A}^{k}\in \mathbb{R}\left( N\times N\otimes n_{k}\times
n_{k}\right) $, i.e. 
\begin{equation}
\mathbf{A}^{k}=\left( A_{pq}^{k}\right) _{p,q=1}^{N};\;\;A_{pq}^{k}=\left(
a_{pqij}^{k}\right) _{i,j=1}^{n_{k}}\in \mathbb{R}\left( n_{k}\times
n_{k}\right) ,  \label{matrix1}
\end{equation}%
similarly, let $\mathbf{b}^{k}$ be an $N\times N$ matrix with $n_{k}$%
-vectors as its elements, in this case, $\mathbf{b}^{k}\in \mathbb{R}\left(
N\times N\otimes n_{k}\right) :$%
\begin{equation*}
\mathbf{b}^{k}=\left( \vec{b}_{pq}^{k}\right)
_{p,q=1}^{N};\;\;b_{pq}^{k}=\left( b_{pqi}^{k}\right) _{i=1}^{n_{k}}\in 
\mathbb{R}^{n_{k}},
\end{equation*}%
and let $\mathbf{c}^{k}$ be an $N\times N$ matrix $\mathbf{c}^{k}=\left(
c_{pq}^{k}\right) _{1\leq p,q\leq N}\in \mathbb{R}\left( N\times N\right) $.
We adopt the following multiplication conventions. Whenever $\mathbf{A}\in 
\mathbb{R}\left( N\times N\otimes n_{k}\times n_{k}\right) $ and $\mathbf{v}%
\in \mathbb{R}\left( N\otimes n_{k}\right) $, then $\mathbf{Av}\in \mathbb{R}%
\left( N\otimes n_{k}\right) $, $\mathbf{bv}\in \mathbb{R}\left( N\right) $,
and they are given by%
\begin{equation*}
\mathbf{Av}=\left( \sum_{q=1}^{N}A_{pq}v_{q}\right) _{p=1}^{N},\qquad 
\mathbf{bv}=\left( \sum_{q=1}^{N}b_{pq}v_{q}\right) _{p=1}^{N},
\end{equation*}%
where $A_{ij}\in \mathbb{R}\left( n_{k}\times n_{k}\right) $, $%
b_{ij},v_{j}\in \mathbb{R}\left( n_{k}\right) $, $i,j=1,\cdots ,N$, Given a
vector function $\mathbf{v}\left( x\right) \in \mathbb{R}\left( N\otimes
n_{k}\right) $, we define the divergence operator $\mathbf{div}_{k}\mathbf{v}%
\in \mathbb{R}\left( N\right) $ as%
\begin{equation*}
\mathbf{div}_{k}\mathbf{v}=\left( \mathrm
{div}_{k}v_{p}\right) _{p=1}^{N}.
\end{equation*}%
With these conventions, we define the systems of linear operators%
\begin{equation*}
\mathbf{L}^{k}\mathbf{u}=-\mathbf{div}_{k}\mathbf{A}^{k}\nabla _{k}\mathbf{u}%
+\mathbf{b}^{k}\nabla _{k}\mathbf{u}+\mathbf{c}^{k}\mathbf{u}.
\end{equation*}%
Notice that $\mathbf{L}^{k}\mathbf{u}\in \mathbb{R}^{N}$, and the principal
part of $\mathbf{L}^{k}$ is%
\begin{equation*}
-\mathbf{div}_{k}\mathbf{A}^{k}\nabla _{k}\mathbf{u}=-\left( \mathrm
{div}%
_{k}\left( \sum_{q=1}^{N}A_{pq}^{k}\nabla _{k}u_{q}\right) \right)
_{p=1}^{N}.
\end{equation*}%
The system $\mathbf{L}^{k}$ may be expressed in terms of the scalar
operators $L_{pq}^{k}$ 
\begin{equation}
L_{pq}^{k}=-\mathrm
{div}_{k}A_{pq}^{k}(x)\nabla _{k}+b_{pq}^{k}(x)\nabla
_{k}+c_{pq}^{k}(x).  \label{fam_s}
\end{equation}%
Indeed, we have that the $p^{\text{th}}$-component of $\mathbf{L}^{k}\mathbf{%
u}$ is $\left( \mathbf{L}^{k}\mathbf{u}\right)
_{p}=\sum_{q=1}^{N}L_{pq}^{k}u_{q}$.

We will assume that each system of operators $\mathbf{L}^{k}$, $1\leq k\leq
m $ is subelliptic in $\mathbb{R}^{n_{k}}$, in the following sense:

\begin{definition}[Subelliptic system]
\label{subel_def_s} Let $\mathbf{L}$ be a linear system given by 
\begin{equation}
\mathbf{Lu}=-\mathbf{divA}\nabla \mathbf{u}+\mathbf{b}\nabla \mathbf{u}+%
\mathbf{cu}  \label{subop_s}
\end{equation}%
where 
\begin{equation*}
\mathbf{A}=\left( A_{pq}\right) _{1\leq p,q\leq N}=\left( a_{pqij}\right) 
_{\substack{ 1\leq p,q\leq N  \\ 1\leq i,j\leq n}}\in \mathbb{R}\left(
N\times N\otimes n\times n\right) ,
\end{equation*}%
$\mathbf{b}\in \mathbb{R}\left( N\times N\otimes n\right) $, and $\mathbf{c}%
\in \mathbb{R}\left( N\times N\right) $. Then $\mathbf{L}$ is \emph{%
subelliptic at $x^{0}\in \mathbb{R}^{n}$} if there exists a neighborhood $U$
of $x^{0}$ and positive constants $\varepsilon $ and $C$ such that 
\begin{equation}
||\mathbf{u}||_{\varepsilon }^{2}\leq C\left\{ \left\vert \left( \mathbf{Lu},%
\mathbf{u}\right) \right\vert +\left\Vert \mathbf{u}\right\Vert ^{2}\right\}
\label{subel_s}
\end{equation}%
for all $\mathbf{u}=(u_{1},u_{2},\ldots ,u_{N})$ such that $u_{i}\in
C_{0}^{\infty }(\mathbb{R}^{n})$, $i=1,\cdots ,N$. $L$ is called \emph{%
subelliptic} if it is subelliptic at each point of $\mathbb{R}^{n}$.
\end{definition}

The main result for systems of equations is the following:

\begin{theorem}
\label{main_s} Let $\lambda _{k}\in C^{\infty }\left( \mathbb{R}^{n}\right) $%
, $1\leq k\leq m$ be such that $\lambda _{1}\equiv 1$ and $\lambda _{k}(%
\overline{x^{k}})>0$ if $\overline{x^{k}}\neq 0$ for$\;k=2,\cdots ,m$. Let
the matrices $\mathbf{A}^{k}$ be symmetric, namely $%
a_{pqij}^{k}=a_{qpij}^{k} $, and assume that for each $1\leq k\leq m$ the
systems $\mathbf{L}^{k}$ are subelliptic in $\mathbb{R}^{n_{k}}$. Then the
operator $\mathbf{L}$ defined by 
\begin{equation}
\mathbf{L}=\sum\limits_{k=1}^{m}\lambda _{k}\mathbf{L}_{k}
\label{operator_s}
\end{equation}%
is hypoelliptic in $\mathbb{R}^{n}$. More precisely, if $\mathbf{u}$ is a
vector of distributions on $\mathbb{R}^{n}$ such that $\zeta L\mathbf{u}\in
\prod_{k=1}^{N}H^{s}(\mathbb{R}^{n})$ for all $\zeta \in C_{0}^{\infty }(U)$
where $U$ is an open set in $\mathbb{R}^{n}$, then $\zeta \mathbf{u}\in
\prod_{k=1}^{N}H^{s}(\mathbb{R}^{n})$ for all $\zeta \in C_{0}^{\infty }(U)$%
. That is, $\mathbf{L}$ is hypoelliptic without loss of derivatives.
\end{theorem}

In this work we broadly follow the line of the proof established by Kohn 
\cite{Kohn}, the presence of more than one function $\lambda _{i}$ prevents
however of a straightforward adaptation of proofs and requires a more
delicate analysis. The paper is organized as follows. First, in Section \ref%
{prelim} we give some preliminary lemmas that are used further in Section %
\ref{Apriori estimates} to prove the main a-priori estimate, Lemma \ref%
{main_apriori}. The main result is proved in Section \ref{hypoellipticity}
using families of smoothing operators.

\section{Preliminaries\label{prelim}}

In this section we give basic definitions and establish some preliminary
results which will be used in our proofs.

The Fourier transform of an integrable function $u$ is defined by 
\begin{equation*}
\hat{u}(\xi )=\int\limits_{\mathbb{R}^{n}}e^{-ix\cdot \xi }u(x)~dx.
\end{equation*}%
The inverse Fourier transform is given by%
\begin{equation*}
f^{\vee }\left( x\right) =\int_{\mathbb{R}^{n}}e^{ix\cdot \xi }f(\xi )~d%
\overline{\xi },
\end{equation*}%
where $d\overline{\xi }=\left( 2\pi \right) ^{-n}d\xi $. Note that $f^{\vee
}\left( x\right) =\left( 2\pi \right) ^{-n}\widehat{\widetilde{f}}\left(
x\right) $ where $\widetilde{f}\left( \xi \right) =f\left( -\xi \right) $.

\begin{definition}
\label{def-Sobolev-norm}For any $s\in \mathbb{R}$ we define an operator $%
\Lambda ^{s}$ by the identity 
\begin{equation}
\widehat{\Lambda ^{s}u}(\xi )=\left( 1+\left\vert \xi \right\vert
^{2}\right) ^{s/2}\hat{u}(\xi )  \label{lambda}
\end{equation}%
and the norm $||\cdot ||_{s}$ by 
\begin{equation}
||u||_{s}=||\Lambda ^{s}u||_{L^{2}(\mathbb{R}^{n})}  \label{norm}
\end{equation}%
For any vector function $\mathbf{u}=(u_{1},\ldots ,u_{N})$ we define $%
\Lambda ^{s}\mathbf{u}$ by the identity $\widehat{\Lambda ^{s}\mathbf{u}}%
=\left( \widehat{\Lambda ^{s}u_{1}},\ldots ,\widehat{\Lambda ^{s}u_{N}}%
\right) $, with the norm%
\begin{equation*}
||\mathbf{u}||_{s}=\left( \sum_{p=1}^{N}||\Lambda ^{s}u_{p}||_{L^{2}(\mathbb{%
R}^{n})}^{2}\right) ^{1/2}.
\end{equation*}
\end{definition}

We recall that, more generally, a pseudodifferential operator $P$ with
symbol $p\left( x,\xi \right) $ is given by%
\begin{equation*}
Pf\left( x\right) =\int_{\mathbb{R}^{n}}e^{ix\cdot \xi }p\left( x,\xi
\right) \widehat{u}(\xi )~d\overline{\xi }.
\end{equation*}%
Note that if $p\left( x,\xi \right) =i\xi _{j}$, then $P=\dfrac{\partial }{%
\partial x_{j}}$.

\begin{definition}
Given $u\in C_{0}^{\infty }(\mathbb{R}^{n})$ define the \emph{partial
Fourier transform} $\mathfrak{F}_{x^{k}}u(\overline{x^{k}},\xi ^{k})$ by 
\begin{equation*}
\mathfrak{F}_{x^{k}}u(\overline{x^{k}},\xi ^{k})=\int\limits_{\mathbb{R}%
_{x^{k}}^{n_{k}}}e^{-ix^{k}\cdot \xi ^{k}}u(x)dx^{k}
\end{equation*}%
For vector functions $\mathbf{u}=(u_{1},\ldots ,u_{N}),\;u_{i}\in
C_{0}^{\infty }(\mathbb{R}^{n})$ we set 
\begin{equation*}
\mathfrak{F}_{x^{k}}\mathbf{u}\left( \overline{x^{k}},\xi ^{k}\right)
=\left( \mathfrak{F}_{x^{k}}u_{1}\left( \overline{x^{k}},\xi ^{k}\right)
,\ldots ,\mathfrak{F}_{x^{k}}u_{N}\left( \overline{x^{k}},\xi ^{k}\right)
\right) .
\end{equation*}
\end{definition}

\begin{definition}
For $s\in \mathbb{R}$ define the partial operators $\Lambda _{x^{k}}^{s}$ by 
\begin{equation*}
\mathfrak{F}_{x^{k}}(\Lambda _{x^{k}}^{s}u)(\overline{x^{k}},\xi
^{k})=(1+|\xi ^{k}|^{2})^{s/2}\mathfrak{F}_{x^{k}}u(\overline{x^{k}},\xi
^{k})
\end{equation*}%
Similarly, for vector functions $\mathbf{u}$, we set $\mathfrak{F}%
_{x^{k}}(\Lambda _{x^{k}}^{s}\mathbf{u})(\overline{x^{k}},\xi ^{k})=(1+|\xi
^{k}|)^{s/2}\mathfrak{F}_{x^{k}}\mathbf{u}(\overline{x^{k}},\xi ^{k})$.
\end{definition}

The next lemma is the classical result on a composition of
pseudodifferential operators (see for example \cite{Tay}). In what follows $%
S^{m}$ denotes the usual classes $S_{1,0}^{m}$, of symbols $p\left( x,\xi
\right) $ satisfying 
\begin{equation}
\left\vert \partial _{x}^{\alpha }\partial _{\xi }^{\beta }p\left( x,\xi
\right) \right\vert \leq B_{\alpha ,\beta }\left( 1+\left\vert \xi
\right\vert ^{2}\right) ^{\frac{1}{2}\left( m-\left\vert \beta \right\vert
\right) },\qquad \text{for all }x,\xi ,\text{ }\alpha ,\beta .
\label{symbols}
\end{equation}

\begin{lemma}
\label{composition}Let $p\left( x,\xi \right) \in S^{m}$ and $q\left( x,\xi
\right) \in S^{k}$ then 
\begin{equation*}
p\left( x,D\right) q\left( x,D\right) =r\left( x,D\right) \in S^{m+k}
\end{equation*}%
with 
\begin{equation*}
r\left( x,\xi \right) \sim \sum\limits_{\alpha \geq 0}\frac{i^{|\alpha |}}{%
\alpha !}D_{\xi }^{\alpha }p\left( x,\xi \right) D_{x}^{\alpha }q\left(
x,\xi \right) ,
\end{equation*}%
in the sense that 
\begin{equation*}
\left( p\left( x,\xi \right) -\sum_{\left\vert \alpha \right\vert <N}D_{\xi
}^{\alpha }p\left( x,\xi \right) D_{x}^{\alpha }q\left( x,\xi \right)
\right) \in S^{m+k-N},\text{ for all }N\geq 0.
\end{equation*}
\end{lemma}

We now give two general lemmas concerning pseudodifferential and subelliptic
operators. The following lemma \cite{Kohn} is a main tool for dealing with
the inner products involving pseudodifferential operators and ordinary
derivatives. Roughly speaking, it allows to lower the order of
differentiation in an inner product using integration by parts and standard
pseudodifferential calculus.

We will localize our estimates by multiplication with suitable cutoff
functions. The following concepts will be useful in our microlocal analysis.

\begin{definition}[Cutoff functons, supporting relation]
We say that ${\varphi }$ is a cutoff function in $\mathbb{R}^{n}$ if ${%
\varphi \in }C_{0}^{\infty }\left( \mathbb{R}^{n}\right) $ and $0\leq {%
\varphi \leq 1}$. Given two measurable functions ${\varphi ,\psi }$ we
introduce the notation ${\varphi }\prec \psi $, and we say that $\psi $ 
\emph{supports} ${\varphi }$ if $\psi $ is a cutoff function and $\psi
\equiv 1$ in a neighbourhood of $\mathrm
{support}{\varphi }$.
\end{definition}

\begin{lemma}
\label{general} Let $P$ and $Q$ be pseudodifferential operators of orders $p$
and $q$, respectively. Assume that $P-P^{\ast }$ and $Q-Q^{\ast }$ are of
orders (at most) $p-1$ and $q-1$, respectively. Let $\zeta ,\eta \in
C_{0}^{\infty }(\mathbb{R}^{n})$, such that $\zeta _{x_{i}}\prec \eta $.
Then there exists $C>0$ such that for all\textbf{\ }$\mathbf{u}\in
\prod_{k=1}^{N}C^{\infty }(\mathbb{R}^{n})$ 
\begin{equation}  \label{p_q_est}
\left\vert \left( P\zeta \mathbf{u}_{x_{i}},Q\zeta \mathbf{u}\right)
\right\vert \leq C\left( \left\Vert \zeta \mathbf{u}\right\Vert
_{(p+q)/2}^{2}+\left\Vert \eta \mathbf{u}\right\Vert _{(p+q)/2}^{2}\right) .
\end{equation}
Moreover, if $\mathbf{u}\in \prod_{k=1}^{N}H^{r }(\mathbb{R}^{n})$ with $%
r=\max\{p+2,q+1\}$, then the same estimate holds.
\end{lemma}

\begin{proof}
For the simplicity of the argument let us consider the scalar case. The
desired estimate has been already shown for $u\in C^{\infty }(\mathbb{R}%
^{n}) $ \cite{Kohn}. In case $u\in H^{r}(\mathbb{R}^{n})$ we find an
approximating sequence $\{u_{n}\}_{n=1}^{\infty }\subset C^{\infty }$ such
that $\lim_{n\rightarrow \infty }||u_{n}-u||_{r}=0$. One can check that $%
u_{n}$ defined by $\widehat{u_{n}}(\xi )=\exp (-|\xi |^{2}/n^{2})\widehat{u}%
(\xi )$ satisfies the desired properties for all $p,\;q\;\in \mathbb{R}$. By
the definition of $r$ it follows that $\lim_{n\rightarrow \infty
}||u_{n}-u||_{(p+q)/2}=0$ and, moreover, by Arzela-Ascoli theorem (replacing 
$\left\{ u_{n}\right\} $ by an appropriate subsequence, which we dub again $%
\left\{ u_{n}\right\} $) $||P\zeta \partial _{x_{i}}(u_{n}-u)||\rightarrow 0$
and $||Q\zeta (u_{n}-u)||\rightarrow 0$ as $n\rightarrow \infty $.
Therefore, applying (\ref{p_q_est}) to $u_{n}$ and taking the limit as $%
n\rightarrow \infty $ we obtain the desired result.
\end{proof}

We will henceforth use special families of cutoff functions satisfying the
following properties.

\begin{itemize}
\item We let $\sigma _{k},\;\tilde{\sigma}_{k},\;\sigma _{k}^{\prime
},\;\sigma _{k}^{\prime \prime }\in C_{0}^{\infty }(\mathbb{R}%
^{n_{k}}),\;k=1,\cdots ,m$ be cutoff functions such that $\sigma _{k}=1$ in
a neighbourhood of $0\in \mathbb{R}^{n_{k}}$, and ${\sigma }_{k}\prec \tilde{%
\sigma}_{k}\prec \sigma _{k}^{\prime }\prec \sigma _{k}^{\prime \prime }$.
Let $\zeta (x)=\prod\limits_{k=1}^{m}\sigma _{k}(x^{k})$, with $\tilde{\zeta}%
(x)$, $\zeta ^{\prime }(x)$, and $\zeta ^{\prime \prime }(x)$ are similarly
defined.

\item We fix $U_{0}^{k}$ and $U^{k}$ to be neighborhoods of the origin in $%
\mathbb{R}^{n_{k}}$ such that $\overline{U}_{0}^{k}\subset U^{k}$ and $%
\sigma _{k}=1$ on $U^{k}$. Let $\sigma _{0}^{k},\;\tilde{\sigma}_{0}^{k}$ be
cutoff functions in $\mathbb{R}^{n_{k}}$ with $\mathrm
{support}(\sigma
_{0}^{k})\cap U_{0}^{k}=\emptyset $, and $\left\vert \nabla _{x^{k}}\sigma
_{k}\right\vert \prec \sigma _{0}^{k}\prec \tilde{\sigma}_{0}^{k}$. Set $%
\zeta _{0}^{k}=\sigma _{0}^{k}\prod\limits_{l=1,l\neq k}^{m}\tilde{\sigma}%
_{l}(x^{l})$, and $\tilde{\zeta}_{0}^{k}=\tilde{\sigma}_{0}^{k}\prod%
\limits_{l=1,l\neq k}^{m}\sigma _{l}^{\prime }(x^{l})$. Note that $\zeta
_{0}^{k}\zeta _{x_{i}^{k}}=\zeta _{x_{i}^{k}}$.

\item We choose the cutoffs functions\ so that they also satisfy $\sigma
_{k}^{0}\prec \tilde{\sigma}_{k}$, $\tilde{\sigma}_{k}^{0}\prec \sigma
_{k}^{\prime }$. Hence $\zeta _{0}^{k}\prec \tilde{\zeta}$ and $\tilde{\zeta}%
_{0}^{k}\prec \zeta ^{\prime }$.

\item In the case $k=1$ we write $\zeta _{0}$ for $\zeta _{0}^{1}$.
\end{itemize}

The next lemma is the classical result on a composition of
pseudodifferential operators (see for example \cite{Tay}). In what follows $%
S^{m}$ denotes the usual classes $S_{1,0}^{m}$, of symbols $p\left( x,\xi
\right) $ satisfying%
\begin{equation*}
\left\vert \partial _{x}^{\xi }\partial _{\xi }^{\beta }p\left( x,\xi
\right) \right\vert \leq B_{\alpha ,\beta }\left( 1+\left\vert \xi
\right\vert ^{2}\right) ^{\frac{1}{2}\left( m-\left\vert \beta \right\vert
\right) },\qquad \text{for all }x,\xi ,\text{ }\alpha ,\beta .
\end{equation*}%
To carry out an approximation scheme we will define a family of smoothing
pseudodifferential operators \cite{Kohn}.

\begin{definition}
For $\delta >0$ we define $S_{\delta }$ by%
\begin{equation}
\widehat{S_{\delta }u}\left( \xi \right) =\frac{1}{\left( 1+\delta
^{2}\left\vert \xi \right\vert ^{2}\right) ^{3/2}} \widehat{u}\left( \xi
\right) =s_{\delta }\left( \xi \right) \widehat{u}\left( \xi \right) .
\label{smooth_oper_S}
\end{equation}
\end{definition}

The operator $S_{\delta }$ is partially smoothing; in particular, if $u\in
H^s(\mathbb{R}^n)$, then $S_{\delta } u\in H^{s+3}\left( \mathbb{R}%
^{n}\right) $. We also have:

\begin{lemma}
\label{properties_S}The operator $S_{\delta }$ has the following properties:%
\renewcommand{\theenumi}{\roman{enumi}}%

\begin{enumerate}
\item \label{0_S}$S_{\delta }\in S^{0}$ uniformly in $\delta $ for $0\leq
\delta \leq 1$, where $S^{0}$ is the symbol class defined by (\ref{symbols})
with $m=0$. More precisely, for any $s\in \mathbb{R}$ 
\begin{equation*}
\sup_{0<\delta \leq 1}\left\Vert S_{\delta }u\right\Vert _{s}=\left\Vert
u\right\Vert _{s}.
\end{equation*}

\item \label{1_S}$S_{\delta }:H^{s}\mapsto H^{s}$ is a bounded operator,
with bounds independent of $\delta $.

\item \label{2_S}If $u\in H^{s_{0}}$ for some $s_{0}\in \mathbb{R}$ then $%
S_{\delta }u\in H^{s_0+3}$.

\item \label{3_S}If for any $s\in \mathbb{R}$, $u\in H^{s-3}$ and $%
\lim_{\delta \rightarrow 0^{+}}\left\Vert S_{\delta }u\right\Vert _{s}\leq C$%
, then $u\in H^{s}$.

\item \label{4_S}If $a\in C_{0}^{\infty }(\mathbb{R}^{n})$, then 
\begin{equation*}
\begin{array}{rcl}
\lbrack a,S_{\delta }] & = & \sum_{k=1}^{m}\sum_{j=1}^{n_{k}}\left(
a_{x_{j}^{k}}\frac{\partial }{\partial x_{j}^{k}}\right) R_{\delta
}^{-2}S_{\delta }+Q_{\delta }^{-2}S_{\delta }%
\end{array}%
\end{equation*}%
where $R_{\delta }^{-2}$ and $Q_{\delta }^{-2}$ are families of
pseudodifferential operators of order $-2$ uniformly in $\delta $ for $0\leq
\delta \leq 1$.
\end{enumerate}
\end{lemma}

\begin{proof}
Since $\exp \frac{1}{\left( 1+\delta ^{2}\left\vert \xi \right\vert
^{2}\right) ^{3/2}}\leq 1$, we have 
\begin{eqnarray*}
\left\Vert S_{\delta }u\right\Vert _{s}^{2} &=&\int_{\mathbb{R}^{n}}\left(
1+\left\vert \xi \right\vert ^{2}\right) ^{s}\left( \frac{1}{\left( 1+\delta
^{2}\left\vert \xi \right\vert ^{2}\right) ^{3/2}}\right) ^{2}\left\vert 
\widehat{u}\right\vert ^{2}~d\xi  \\
&\leq &\int_{\mathbb{R}^{n}}\left( 1+\left\vert \xi \right\vert ^{2}\right)
^{s}\left\vert \widehat{u}\right\vert ^{2}~d\xi =\left\Vert u\right\Vert
_{s}^{2}
\end{eqnarray*}%
This proves (\ref{0_S}) and (\ref{1_S}). Property (\ref{2_S}) follows easily
from the definition of $S_{\delta }$. 

On the other hand, if $u\in H^{s-3}$ and $\lim_{\delta \rightarrow
0^{+}}\left\Vert S_{\delta }u\right\Vert _{s}\leq C$ then%
\begin{eqnarray*}
\left\Vert u\right\Vert _{s} &=&\int_{\mathbb{R}^{n}}\left( 1+\left\vert \xi
\right\vert ^{2}\right) ^{s}\left\vert \widehat{u}\right\vert ^{2}~d\xi  \\
&=&\lim_{\delta \rightarrow 0^{+}}\int_{\mathbb{R}^{n}}\left( 1+\left\vert
\xi \right\vert ^{2}\right) ^{s}\left( \frac{1}{\left( 1+\delta
^{2}\left\vert \xi \right\vert ^{2}\right) ^{3/2}}\right) ^{2}\left\vert 
\widehat{u}\right\vert ^{2}~d\xi  \\
&=&\lim_{\delta \rightarrow 0^{+}}\left\Vert \Lambda _{s}S_{\delta
}u\right\Vert ^{2}\leq C^{2}.
\end{eqnarray*}
So $||u||_{s}<C,$ which shows (\ref{3_S}). To prove property (\ref{4_S}) we
first note that has been shown in \cite[Lemma 3.3]{Kohn} using Lemma \ref%
{composition} that the principal symbol of $[a,S_{\delta }]$ has the form 
\begin{equation*}
\lbrack a,S_{\delta }]\sim
\sum_{k=1}^{m}\sum_{j=1}^{n_{k}}-2ia_{x_{j}^{k}}\xi _{j}^{k}\frac{3/2\delta
^{2}}{1+\delta ^{2}|\xi |^{2}}S_{\delta }
\end{equation*}%
By differentiating $\xi _{j}^{k}\frac{3/2\delta ^{2}}{1+\delta ^{2}|\xi |^{2}%
}S_{\delta }$ with respect to $\xi _{r}^{l}$ one can check that the lower
order terms have the form $Q_{\delta }^{-2}S_{\delta }$.
\end{proof}

\section{Apriori estimates\label{Apriori estimates}}

In what follows we establish estimates both for scalar functions $u$ and
vector functions $\mathbf{u}$, as well as scalar operators $L$ (\ref%
{operator}) and linear systems of operators $\mathbf{L}$ (\ref{operator_s}).
We will state the results for systems of equations, with the understanding
that linear equations correspond to the case $N=1$. The proofs for the
scalar or systems cases do not differ in any substantial way, so, for
simplicity, we only include the proof for the scalar case.

The next lemma gives a useful estimate for subelliptic operators. It follows
directly from the definition of subellipticity (\ref{subel}) with the help
of Lemma \ref{general}.

\begin{lemma}
\label{subel_lem} Suppose that $L$, $\mathbf{L}$, given by (\ref{subop}), (%
\ref{subop_s}) respectively, are subelliptic at $x_{0}$ and that $\zeta ,%
\tilde{\zeta}\in C_{0}^{\infty }\left( U\right) $, with $U$ a neighbourhood
of $x_{0}$ as in (\ref{subel}), and $\zeta \prec \tilde{\zeta}$. Then, for
all $\mathbf{u}\in C^{\infty }\left( \mathbb{R}^{n}\right) $,%
\begin{eqnarray}
\left\Vert \zeta \mathbf{u}\right\Vert _{\varepsilon }^{2} &\leq &C\left\{
\left\vert \sum_{p,q=1}^{N}\sum_{i,j=1}^{n}\left( a_{pqij}\zeta \left(
u_{q}\right) _{x_{i}},\zeta \left( u_{p}\right) _{x_{j}}\right) \right\vert
+\left\Vert \tilde{\zeta}\mathbf{u}\right\Vert ^{2}\right\}  \label{epsnorm}
\\
&\leq &C\left\{ \left\vert \left( \zeta \mathbf{Lu},\zeta \mathbf{u}\right)
\right\vert +\left\Vert \tilde{\zeta}\mathbf{u}\right\Vert ^{2}\right\} . 
\notag
\end{eqnarray}
\end{lemma}

\begin{proof}
We consider only the scalar case $N=1$. From (\ref{subel}) we have 
\begin{equation}
\left\Vert \zeta u\right\Vert _{\varepsilon }^{2}\leq C\{\left\vert \left(
L\zeta u,\zeta u\right) \right\vert +\left\Vert \zeta u\right\Vert ^{2}\}.
\label{ineq-subellem-00}
\end{equation}%
Next,%
\begin{equation}
\left( L\zeta u,\zeta u\right) =-\sum\limits_{1}^{n}\left( a_{ij}(\zeta
u)_{x_{i}x_{j}},\zeta u\right) +\sum\limits_{1}^{n}\left( b^{i}(\zeta
u)_{x_{i}},\zeta u\right) +\left( c\zeta u,\zeta u\right) .
\label{ineq-subellem-01}
\end{equation}%
Integrating by parts the first term on the right, we have that%
\begin{eqnarray*}
-\left( a_{ij}(\zeta u)_{x_{i}x_{j}},\zeta u\right)
&=&\sum\limits_{1}^{n}\left( \left( a_{ij}\right) _{x_{j}}(\zeta
u)_{x_{i}},\zeta u\right) +\sum\limits_{1}^{n}\left( a_{ij}(\zeta
u)_{x_{i}},\left( \zeta u\right) _{x_{j}}\right) \\
&=&\sum\limits_{1}^{n}\left( \left( a_{ij}\right) _{x_{j}}(\zeta
u)_{x_{i}},\zeta u\right) +2\sum\limits_{1}^{n}\left( a_{ij}\zeta
u_{x_{j}},\zeta _{x_{i}}u\right) \\
&&+\sum\limits_{1}^{n}\left( a_{ij}\zeta _{x_{i}}u,\zeta _{x_{j}}u\right)
+\sum\limits_{1}^{n}\left( a_{ij}\zeta u_{x_{i}},\zeta u_{x_{j}}\right) .
\end{eqnarray*}%
By Lemma \ref{general} the first and second terms on the right are bounded
by $C||\tilde{\zeta}u||^{2}$, while it is clear that the third term on the
right also satisfies the same bounds. The same applies to the last two terms
on the right of (\ref{ineq-subellem-01}), This and (\ref{ineq-subellem-00})
yield 
\begin{equation*}
\left\Vert \zeta u\right\Vert _{\varepsilon }^{2}\leq C\{\left\vert
\sum\limits_{i,j=1}^{n}\left( a_{ij}\zeta u_{x_{i}},\zeta u_{x_{j}}\right)
\right\vert +\left\Vert \zeta u\right\Vert ^{2}\}.
\end{equation*}%
prove the first inequality in (\ref{epsnorm}).

To prove the second inequality in (\ref{epsnorm}), integrating by parts we
write%
\begin{eqnarray*}
\sum \left( {a}_{ij}{\zeta u_{x_{i}},\zeta u_{x_{j}}}\right) &=&-\sum \left( 
{a}_{ij}{\zeta u_{x_{i}x_{j}},\zeta u}\right) -\sum \left( \left( {a}_{ij}{%
\zeta }^{2}\right) _{x_{j}}\tilde{\zeta}{u_{x_{i}},}\tilde{\zeta}{u}\right)
\\
&=&\left( \zeta Lu,\zeta u\right) -\left( c\zeta u,\zeta u\right)
-\sum\limits_{1}^{n}\left( b^{i}\zeta u_{x_{i}},\zeta u\right) \\
&&-\sum \left( \left( {a}_{ij}{\zeta }^{2}\right) _{x_{j}}\tilde{\zeta}{%
u_{x_{i}},}\tilde{\zeta}{u}\right) ,
\end{eqnarray*}%
and apply Lemma \ref{general} the the last two terms on the right.
\end{proof}

Next, we will establish a number of auxiliary results which will be used to
prove the main a priori estimate and the main theorem.

The following lemma is a generalization of Lemma \ref{subel_lem} for the
operators and systems of operators of the form (\ref{operator}), (\ref%
{operator_s}).

\begin{lemma}
\label{inter_lem}Let $\mathbf{L}$ be defined by (\ref{operator_s}) (or by (%
\ref{operator}) when $N=1$), with $\mathbf{L}_{k}$ subelliptic at $%
x_{0}^{k}\in \mathbb{R}^{n_{k}}$ and $0\leq \lambda _{k}\in C^{\infty
}\left( \mathbb{R}^{n_{k}}\right) $, $k=1,\cdots ,m$. Let $U_{k}\subset 
\mathbb{R}^{n_{k}}$ be neighbourhoods of $x_{0}^{k}$ such that (\ref{subel_s}%
) holds with $\varepsilon _{k}$ for $\mathbf{L}_{k}$ in $U_{k}$ (resp. (\ref%
{subel}) holds with $\varepsilon _{k}$ for $L_{k}$ in $U_{k}$). Then if $%
\zeta ,\tilde{\zeta}\in C_{0}^{\infty }\left( \prod_{k=1}^{m}U_{k}\right) $,
with $\zeta \prec \tilde{\zeta}$, for $\varepsilon =\min_{k}\varepsilon _{k}$
we have%
\begin{equation}
\begin{tabular}{rcl}
$\sum\limits_{k=1}^{m}\left\Vert \sqrt{\lambda _{k}}\Lambda
_{x^{k}}^{\varepsilon }\left( \zeta \mathbf{u}\right) \right\Vert ^{2}$ & $%
\leq $ & $C\sum\limits_{k=1}^{m}\sum\limits_{i,j=1}^{n_{k}}\sum%
\limits_{p,q=1}^{N}\left( \zeta \lambda _{k}a_{pqij}^{k}\left( u_{q}\right)
_{x_{i}^{k}},\zeta \left( u_{p}\right) _{x_{j}^{k}}\right) $ \\ 
&  & +$C\left\Vert \tilde{\zeta}\mathbf{u}\right\Vert ^{2}$ \\ 
& $\leq $ & $C\left\{ \left\vert \left( \zeta \mathbf{Lu},\zeta \mathbf{u}%
\right) \right\vert +\left\Vert \tilde{\zeta}\mathbf{u}\right\Vert
^{2}\right\} .$%
\end{tabular}
\label{est1}
\end{equation}%
Moreover, the same estimate holds when $\mathbf{u}\in \prod_{k=1}^{N}H^{2}(%
\mathbb{R}^{n})$.
\end{lemma}

\begin{proof}
It is enough to consider the scalar case $N=1$. First, consider $\mathbf{u}%
\in \prod_{k=1}^{N}C^{\infty }(\mathbb{R}^{n})$. Since for each $k$ the
operator $L_{k}$ is subelliptic it follows from (\ref{epsnorm}) that for
each fixed $\overline{x^{k}}\in \mathbb{R}^{n_{k}}$ we have%
\begin{equation*}
\begin{split}
\lambda _{k}& (\overline{x^{k}})\int\limits_{\mathbb{R}^{n_{k}}}|\Lambda
_{x^{k}}^{\varepsilon }(\zeta u)(x)|^{2}\,dx^{k} \\
& \leq C\sum\limits_{i,j=1}^{n_{k}}\int\limits_{\mathbb{R}^{n_{k}}}\lambda
_{k}(\overline{x^{k}})a_{ij}^{k}(x^{k})\zeta (x)u_{x_{i}^{k}}(x)\zeta
(x)u_{x_{j}^{k}}(x)\,dx^{k} \\
& +C\int_{\mathbb{R}^{n_{k}}}\left\vert \tilde{\zeta}\left( x\right) u\left(
x\right) \right\vert ^{2}~dx^{k}.
\end{split}
\label{int}
\end{equation*}%
Integrating the above inequality with respect to $\overline{x^{k}}$ and
summing over $k=1,\cdots ,m$ we obtain the first part of (\ref{est1}). To
show that the inequality holds for $\mathbf{u}\in \prod_{k=1}^{N}H^{2 }(%
\mathbb{R}^{n})$ we perform an approximation in the same way it has been
done in the proof of Lemma \ref{general}.

To prove the second part consider each term of the triple sum in the first
inequality of the lemma and integrate by parts 
\begin{equation}
(\lambda _{k}a_{ij}^{k}\zeta u_{x_{i}^{k}},\zeta u_{x_{j}^{k}})=-\left(
\lambda _{k}a_{ij}^{k}\zeta u_{x_{i}^{k}x_{j}^{k}},\zeta u\right) -\left(
\left( \lambda _{k}a_{ij}^{k}\zeta ^{2}\right) _{x_{j}^{k}}u_{x_{i}^{k}},%
\tilde{\zeta}u\right) .  \label{parts1}
\end{equation}%
We then have%
\begin{eqnarray*}
\sum_{k=1}^{m}\sum_{i,j=1}^{n_{k}}\left( \lambda _{k}a_{ij}^{k}\zeta
u_{x_{i}^{k}},\zeta u_{x_{j}^{k}}\right) &=&\left( \zeta Lu,\zeta u\right)
-\sum_{k=1}^{m}\sum_{i=1}^{n_{k}}\left( \lambda _{k}b_{i}^{k}\zeta
u_{x_{i}^{k}},\zeta u\right) \\
&&-\sum_{k=1}^{m}\left( \lambda _{k}c^{k}\zeta u,\zeta u\right)
-\sum_{k=1}^{m}\sum_{i,j=1}^{n_{k}}\left( \left( \lambda _{k}a_{ij}^{k}\zeta
^{2}\right) _{x_{j}^{k}}u_{x_{i}^{k}},\tilde{\zeta}u\right) .
\end{eqnarray*}%
The second inequality of the Lemma \ref{inter_lem} then follows from Lemma %
\ref{general}.
\end{proof}

We now formulate the main technical result which allows us to deal with
terms involving commutators $[L,\Lambda ^{s}\zeta ]$. The proof relies on
Lemma \ref{general} and Lemma \ref{composition}.

\begin{lemma}
\label{hard}Given $s\in \mathbb{R}$, there exists $C>0$ such that%
\begin{equation*}
\left\vert \left( \tilde{\zeta}\left[ \mathbf{L},\Lambda ^{s}S_{\delta
}\zeta \right] \mathbf{u},\tilde{\zeta}\Lambda ^{s}S_{\delta }\zeta \mathbf{u%
}\right) \right\vert \leq C\left\{ \left\Vert \tilde{\zeta}S_{\delta }\zeta 
\mathbf{u}\right\Vert _{s}^{2}+\sum\limits_{k=1}^{m}\left\Vert \lambda
_{k}\zeta _{0}^{k}S_{\delta }\zeta _{0}^{k}\mathbf{u}\right\Vert
_{s}^{2}+\left\Vert \zeta ^{\prime }S_{\delta }\zeta ^{\prime }\mathbf{u}%
\right\Vert _{s-1/2}^{2}\right\}
\end{equation*}%
for all functions $\mathbf{u}\in \prod_{k=1}^{N}H^{s-2}(\mathbb{R}^{n})$ and
all $0<\delta \leq 1$. Here $\zeta ,\tilde{\zeta}$, and $\zeta _{0}$ are the
cutoff functions defined above.
\end{lemma}

\begin{proof}
Again, it is enough to consider the scalar case $N=1$. We have%
\begin{eqnarray}
\left[ L,\Lambda ^{s}S_{\delta }\zeta \right] u &=&\left[ L,\Lambda
^{s}S_{\delta }\right] \zeta u+\Lambda ^{s}S_{\delta }\left[ L,\zeta \right]
u  \notag \\
&=&\sum\limits_{k=1}^{m}\left[ \lambda _{k}L_{k},\Lambda ^{s}S_{\delta }%
\right] \zeta u+\sum\limits_{k=1}^{m}\Lambda ^{s}S_{\delta }\left[ \lambda
_{k}L_{k},\zeta \right] u.  \label{triv}
\end{eqnarray}%
Next, for any $k=1,\cdots ,m$,%
\begin{equation}
\begin{tabular}{rcl}
$\left[ \lambda _{k}L_{k},\Lambda ^{s}S_{\delta }\right] $ & $=$ & $%
-\sum\limits_{i,j=1}^{n_{k}}\left[ \lambda _{k}a_{ij}^{k},\Lambda
^{s}S_{\delta }\right] \dfrac{\partial ^{2}}{\partial x_{i}^{k}\partial
x_{j}^{k}}$ \\ 
&  & $+\sum\limits_{i=1}^{n_{k}}\left[ \lambda _{k}b_{i}^{k},\Lambda
^{s}S_{\delta }\right] \dfrac{\partial }{\partial x_{i}^{k}}+\left[ \lambda
_{k}c^{k},\Lambda ^{s}S_{\delta }\right] .$%
\end{tabular}
\label{comm1}
\end{equation}%
and%
\begin{equation}
\begin{tabular}{rcl}
$\left[ \lambda _{k}L_{k},\zeta \right] $ & $=$ & $\lambda _{k}\left[
L_{k},\zeta \right] $ \\ 
& $=$ & $\lambda _{k}\left( -\sum\limits_{i,j=1}^{n_{k}}a_{ij}^{k}\zeta
_{x_{i}^{k}x_{j}^{k}}+\sum\limits_{i=1}^{n_{k}}b_{i}^{k}\zeta
_{x_{i}^{k}}\right) -2\lambda _{k}\sum\limits_{i,j=1}^{n_{k}}a_{ij}^{k}\zeta
_{x_{i}^{k}}\dfrac{\partial }{\partial x_{j}^{k}},$%
\end{tabular}
\label{comm2}
\end{equation}%
where we used the symmetry of $a^{k}$ on the last term. It follows that%
\begin{eqnarray}
&&\left( \tilde{\zeta}\left[ L,\Lambda ^{s}S_{\delta }\zeta \right] u,\tilde{%
\zeta}\Lambda ^{s}S_{\delta }\zeta u\right)   \notag \\
&=&-\sum\limits_{k=1}^{m}\sum\limits_{i,j=1}^{n_{k}}\left( \tilde{\zeta}%
\left[ \lambda _{k}a_{ij}^{k},\Lambda ^{s}S_{\delta }\right] \dfrac{\partial
^{2}}{\partial x_{i}^{k}\partial x_{j}^{k}}\zeta u,\tilde{\zeta}\Lambda
^{s}S_{\delta }\zeta u\right)   \notag \\
&&+\sum\limits_{k=1}^{m}\sum\limits_{i=1}^{n_{k}}\left( \tilde{\zeta}\left[
\lambda _{k}b_{i}^{k},\Lambda ^{s}S_{\delta }\right] \dfrac{\partial }{%
\partial x_{i}^{k}}\zeta u,\tilde{\zeta}\Lambda ^{s}S_{\delta }\zeta
u\right)   \notag \\
&&+\sum\limits_{k=1}^{m}\left( \tilde{\zeta}\left[ \lambda _{k}c^{k},\Lambda
^{s}S_{\delta }\right] \zeta u,\tilde{\zeta}\Lambda ^{s}S_{\delta }\zeta
u\right)   \notag \\
&&-\sum\limits_{k=1}^{m}\sum\limits_{i,j=1}^{n_{k}}\left( \tilde{\zeta}%
\Lambda ^{s}S_{\delta }\lambda _{k}a_{ij}^{k}\zeta _{x_{i}^{k}x_{j}^{k}}u,%
\tilde{\zeta}\Lambda ^{s}S_{\delta }\zeta u\right)   \notag \\
&&+\sum\limits_{k=1}^{m}\sum\limits_{i=1}^{n_{k}}\left( \tilde{\zeta}\Lambda
^{s}S_{\delta }\lambda _{k}b_{i}^{k}\zeta _{x_{i}^{k}}u,\tilde{\zeta}\Lambda
^{s}S_{\delta }\zeta u\right)   \notag \\
&&-2\sum\limits_{k=1}^{m}\sum\limits_{i,j=1}^{n_{k}}\left( \tilde{\zeta}%
\Lambda ^{s}S_{\delta }\lambda _{k}a_{ij}^{k}\zeta _{x_{i}^{k}}\dfrac{%
\partial }{\partial x_{j}^{k}}u,\tilde{\zeta}\Lambda ^{s}S_{\delta }\zeta
u\right)   \notag \\
&=&I+II+III+IV+V+VI.  \label{comm_lemm-00}
\end{eqnarray}%
Using Lemma \ref{composition} and property (\ref{4_S}) of the operator $%
S_{\delta }$, for a smooth function $f$ we have%
\begin{equation}
\left[ \Lambda ^{s}S_{\delta },f\right] =\widetilde{R}_{s,\delta
}^{0}\Lambda ^{s-1}f_{0}S_{\delta }=R_{s,\delta }^{0}\Lambda
^{s-2}\sum_{r=1}^{m}\sum_{\ell =1}^{n_{k}}f_{x_{\ell }^{r}}\frac{\partial }{%
\partial x_{\ell }^{r}}S_{\delta }+Q_{s,\delta }^{-0}\Lambda
^{s-2}f_{0}S_{\delta },  \label{comm_lemm1}
\end{equation}%
where $\widetilde{R}_{s,\delta }^{0},R_{s,\delta }^{0},Q_{s,\delta }^{0}$
are pseudodifferential operators of order $0$ uniformly in $0\leq \delta
\leq 1$, and $f_{0}$ is a cutoff function such that $f_{0}\succ \left\vert
\nabla f\right\vert $. Moreover, $R_{s,\delta }^{0}-\left( R_{s,\delta
}^{0}\right) ^{\ast }$ is of degree $-1$ uniformly in $0\leq \delta \leq 1$.

We use this to estimate the first term on the right in (\ref{comm1}), we
obtain%
\begin{equation}
\begin{tabular}{rcl}
$\left\vert I\right\vert $ & $\leq $ & $\sum\limits_{k,r=1}^{m}\sum%
\limits_{i,j,\ell =1}^{n_{k}}\left\vert \left( \tilde{\zeta}R_{s,\delta
}^{0}\Lambda ^{s-2}\left( \lambda _{k}a_{ij}^{k}\right) _{x_{\ell }^{r}}%
\dfrac{\partial ^{2}}{\partial x_{i}^{k}\partial x_{j}^{k}}\frac{\partial }{%
\partial x_{\ell }^{r}}S_{\delta }\zeta u,\tilde{\zeta}\Lambda ^{s}S_{\delta
}\zeta u\right) \right\vert $ \\ 
&  & $+\sum\limits_{k=1}^{m}\sum\limits_{i,j=1}^{n_{k}}\left\vert \left( 
\tilde{\zeta}Q_{s,\delta }^{-0}\Lambda ^{s-2}\left( \lambda
_{k}a_{ij}^{k}\right) _{0}\dfrac{\partial ^{2}}{\partial x_{i}^{k}\partial
x_{j}^{k}}S_{\delta }\zeta u,\tilde{\zeta}\Lambda ^{s}S_{\delta }\zeta
u\right) \right\vert .$%
\end{tabular}
\label{comm_lemm4}
\end{equation}%
We apply to each term on the the first sum Lemma \ref{general} with 
\begin{equation*}
P=\tilde{\zeta}R_{s,\delta }^{0}\Lambda ^{s-2}\left( \lambda
_{k}a_{ij}^{k}\right) _{x_{\ell }^{r}}\dfrac{\partial ^{2}}{\partial
x_{i}^{k}\partial x_{j}^{k}},\qquad \text{and}\qquad Q=\tilde{\zeta}\Lambda
^{s}.
\end{equation*}
Note, that both operators $P$ and $Q$ have order $s$, and since $u\in
H^{s-2} $ it follows that $S_\delta u\in H^{s+1}$ and therefore Lemma \ref%
{general} is applicable. Since the second term on the right of (\ref%
{comm_lemm4}) is dominated by $C\left\{ \left\Vert \tilde{\zeta}S_{\delta
}\zeta u\right\Vert _{s}^{2}+\left\Vert S_{\delta }\zeta u\right\Vert
_{s-1}^{2}\right\} $, we obtain 
\begin{equation}
\left\vert I\right\vert \leq C\left( \left\Vert \tilde{\zeta}S_{\delta
}\zeta u\right\Vert _{s}^{2}+\left\Vert \zeta ^{\prime }S_{\delta }\zeta
u\right\Vert _{s-1/2}^{2}\right) .  \label{comm_lemm-I}
\end{equation}

By the first identity in (\ref{comm_lemm1}) it follows that%
\begin{eqnarray}
\left\vert II\right\vert +\left\vert III\right\vert &\leq
&\sum\limits_{k=1}^{m}\sum\limits_{i=1}^{n_{k}}\left\vert \left( \tilde{\zeta%
}\widetilde{R}_{s,\delta }^{0}\Lambda ^{s-1}\left( \lambda
_{k}b_{i}^{k}\right) _{0}\dfrac{\partial }{\partial x_{i}^{k}}S_{\delta
}\zeta u,\tilde{\zeta}\Lambda ^{s}S_{\delta }\zeta u\right) \right\vert 
\notag \\
&&+\sum\limits_{k=1}^{m}\left\vert \left( \tilde{\zeta}\widetilde{R}%
_{s,\delta }^{0}\Lambda ^{s-1}\left( \lambda _{k}c^{k}\right) _{0}S_{\delta
}\zeta u,\tilde{\zeta}\Lambda ^{s}S_{\delta }\zeta u\right) \right\vert 
\notag \\
&\leq &C\left( \left\Vert \tilde{\zeta}S_{\delta }\zeta u\right\Vert
_{s}^{2}+\left\Vert \zeta ^{\prime }S_{\delta }\zeta u\right\Vert
_{s-1}^{2}\right) .  \label{comm_lemm-II-III}
\end{eqnarray}

Using that $\zeta _{x_{i}^{k}x_{j}^{k}}=\zeta _{0}^{k}$ $\zeta
_{x_{i}^{k}x_{j}^{k}}$ and $\zeta _{x_{i}^{k}}=\zeta _{0}^{k}\zeta
_{x_{i}^{k}}$, and (\ref{comm_lemm1}), it follows that%
\begin{eqnarray}
\left\vert IV\right\vert +\left\vert V\right\vert &\leq
&\sum\limits_{k=1}^{m}\sum\limits_{i,j=1}^{n_{k}}\left\vert \left( \tilde{%
\zeta}\Lambda ^{s}S_{\delta }a_{ij}^{k}\zeta _{x_{i}^{k}x_{j}^{k}}\lambda
_{k}\zeta _{0}^{k}u,\tilde{\zeta}\Lambda ^{s}S_{\delta }\zeta u\right)
\right\vert  \notag \\
&&+\sum\limits_{k=1}^{m}\sum\limits_{i=1}^{n_{k}}\left\vert \left( \tilde{%
\zeta}\Lambda ^{s}S_{\delta }b_{i}^{k}\zeta _{x_{i}^{k}}\lambda _{k}\zeta
_{0}^{k}u,\tilde{\zeta}\Lambda ^{s}S_{\delta }\zeta u\right) \right\vert 
\notag \\
&\leq &C\left( \left\Vert \tilde{\zeta}S_{\delta }\zeta u\right\Vert
_{s}^{2}+\sum\limits_{k=1}^{m}\left\Vert \lambda _{k}\zeta _{0}^{k}S_{\delta
}\zeta _{0}^{k}u\right\Vert _{s}^{2}+\left\Vert \tilde{\zeta}S_{\delta }%
\tilde{\zeta}u\right\Vert _{s-1}^{2}\right) .  \label{comm_lemm-IV-V}
\end{eqnarray}

Now, for each term in $VI$ we commute the functions $\lambda
_{k}a_{ij}^{k}\zeta _{x_{i}^{k}}$ and $\zeta $, and carry out an integration
by parts. We obtain the identity%
\begin{eqnarray}
~VI^{kij} &=&2\left( \tilde{\zeta}\Lambda ^{s}S_{\delta }\lambda
_{k}a_{ij}^{k}\zeta _{x_{i}^{k}}\dfrac{\partial }{\partial x_{j}^{k}}u,%
\tilde{\zeta}\Lambda ^{s}S_{\delta }\zeta u\right)  \notag \\
&=&-\left( \tilde{\zeta}\Lambda ^{s}S_{\delta }\zeta u,\tilde{\zeta}\Lambda
^{s}S_{\delta }\left( \lambda _{k}a_{ij}^{k}\zeta _{x_{i}^{k}}\right)
_{x_{j}^{k}}u\right)  \label{VI-kij} \\
&&-2\left( \tilde{\zeta}\Lambda ^{s}S_{\delta }\zeta u,\tilde{\zeta}%
_{x_{j}^{k}}\Lambda ^{s}S_{\delta }\lambda _{k}a_{ij}^{k}\zeta
_{x_{i}^{k}}u\right)  \notag \\
&&-\left( \tilde{\zeta}\Lambda ^{s}S_{\delta }\zeta _{x_{j}^{k}}u,\tilde{%
\zeta}\Lambda ^{s}S_{\delta }\lambda _{k}a_{ij}^{k}\zeta _{x_{i}^{k}}u\right)
\notag \\
&&+\left( \tilde{\zeta}\left[ \zeta ,\Lambda ^{s}S_{\delta }\right] \tilde{%
\zeta}\dfrac{\partial }{\partial x_{j}^{k}}u,\tilde{\zeta}\Lambda
^{s}S_{\delta }\lambda _{k}a_{ij}^{k}\zeta _{x_{i}^{k}}u\right)  \notag \\
&&+\left( \tilde{\zeta}\Lambda ^{s}S_{\delta }\tilde{\zeta}\dfrac{\partial }{%
\partial x_{j}^{k}}u,\tilde{\zeta}\left[ \Lambda ^{s}S_{\delta },\zeta %
\right] \lambda _{k}a_{ij}^{k}\zeta _{x_{i}^{k}}u\right)  \notag \\
&&+\left( \tilde{\zeta}\Lambda ^{s}S_{\delta }\tilde{\zeta}\dfrac{\partial }{%
\partial x_{j}^{k}}u,\tilde{\zeta}\left[ \lambda _{k}a_{ij}^{k}\zeta
_{x_{i}^{k}},\Lambda ^{s}S_{\delta }\right] \zeta u\right)  \notag \\
&&+\left( \tilde{\zeta}\left[ \Lambda ^{s}S_{\delta },\lambda
_{k}a_{ij}^{k}\zeta _{x_{i}^{k}}\right] \tilde{\zeta}\dfrac{\partial }{%
\partial x_{j}^{k}}u,\tilde{\zeta}\Lambda ^{s}S_{\delta }\zeta u\right) 
\notag \\
&=&VI_{1}^{kij}+VI_{2}^{kij}+\cdots +VI_{7}^{kij}.  \notag
\end{eqnarray}%
We now consider each term. We have%
\begin{eqnarray}
\left\vert VI_{1}^{kij}\right\vert &\leq &\left\vert \left( \tilde{\zeta}%
\Lambda ^{s}S_{\delta }\zeta u,\tilde{\zeta}\Lambda ^{s}S_{\delta }\lambda
_{k}\left( a_{ij}^{k}\zeta _{x_{i}^{k}}\right) _{x_{j}^{k}}u\right)
\right\vert  \notag \\
&&+\left\vert \left( \tilde{\zeta}\Lambda ^{s}S_{\delta }\zeta u,\tilde{\zeta%
}\Lambda ^{s}S_{\delta }\left( \lambda _{k}\right)
_{x_{j}^{k}}a_{ij}^{k}\zeta _{x_{i}^{k}}u\right) \right\vert  \notag \\
&\leq &C\left\{ \left\Vert \tilde{\zeta}S_{\delta }\zeta u\right\Vert
_{s}^{2}+\left\Vert \zeta ^{\prime }S_{\delta }\zeta u\right\Vert
_{s-1}^{2}+\left\Vert \lambda _{k}\zeta _{0}^{k}S_{\delta }\zeta
_{0}^{k}u\right\Vert _{s}^{2}\right.  \label{VI-1} \\
&&+\left. \left\Vert \left( \lambda _{k}\right) _{x_{j}^{k}}\zeta
_{0}^{k}S_{\delta }\zeta _{x_{i}^{k}}u\right\Vert _{s}^{2}+\left\Vert \tilde{%
\zeta}S_{\delta }\tilde{\zeta}u\right\Vert _{s-1}^{2}\right\} .  \notag
\end{eqnarray}%
We will use the following Wirtinger-type inequality (see e.g. Appendix in 
\cite{Saw}): If $\phi \in C^{2}\left( U\right) $ with $U$ open in $\mathbb{R}%
^{n}$, $\phi $ nonnegative, then for any compact subset $F\subset U$ there
exists a constant $C$ depending on $\left\Vert D^{2}{\phi }\right\Vert
_{L^{\infty }\left( V\right) }$, with $V$ open and $F\subset V\Subset U$,
and $\mathrm
{dist}\left( F,\partial V\right)>0 $ such that 
\begin{equation}
\left\vert D\phi \left( x\right) \right\vert ^{2}\leq C\phi \left( x\right) .
\label{ineq-Wirtinger}
\end{equation}%
We consider the penultimate term in (\ref{VI-1}),%
\begin{equation*}
\left\Vert \left( \lambda _{k}\right) _{x_{j}^{k}}\zeta _{0}^{k}S_{\delta
}\zeta _{x_{i}^{k}}u\right\Vert _{s}^{2}=\left( \Lambda ^{s}\left( \lambda
_{k}\right) _{x_{j}^{k}}\zeta _{0}^{k}S_{\delta }\zeta _{x_{i}^{k}}u,\Lambda
^{s}\left( \lambda _{k}\right) _{x_{j}^{k}}\zeta _{0}^{k}S_{\delta }\zeta
_{x_{i}^{k}}u\right) .
\end{equation*}%
We commute $\zeta _{x_{i}^{k}}$ from the right into the left and $\left(
\lambda _{k}\right) _{x_{j}^{k}}$ from the left into the right. We obtain%
\begin{eqnarray}
&&\left\Vert \left( \lambda _{k}\right) _{x_{j}^{k}}\zeta _{0}^{k}S_{\delta
}\zeta _{x_{i}^{k}}u\right\Vert _{s}^{2}  \notag \\
&=&\left( \Lambda ^{s}\zeta _{0}^{k}S_{\delta }\left( \zeta
_{x_{i}^{k}}\right) ^{2}u,\Lambda ^{s}\left( \left( \lambda _{k}\right)
_{x_{j}^{k}}\right) ^{2}\zeta _{0}^{k}S_{\delta }\zeta _{0}^{k}u\right) 
\notag \\
&&+\left( \Lambda ^{s}\zeta _{0}^{k}S_{\delta }\left( \zeta
_{x_{i}^{k}}\right) ^{2}u,\left[ \left( \lambda _{k}\right)
_{x_{j}^{k}},\Lambda ^{s}\right] \left( \lambda _{k}\right)
_{x_{j}^{k}}\zeta _{0}^{k}S_{\delta }\zeta _{0}^{k}u\right)  \notag \\
&&+\left( \left[ \Lambda ^{s},\left( \lambda _{k}\right) _{x_{j}^{k}}\right]
\zeta _{0}^{k}S_{\delta }\left( \zeta _{x_{i}^{k}}\right) ^{2}u,\Lambda
^{s}\left( \lambda _{k}\right) _{x_{j}^{k}}\zeta _{0}^{k}S_{\delta }\zeta
_{0}^{k}u\right)  \notag \\
&&+\left( \Lambda ^{s}\left( \lambda _{k}\right) _{x_{j}^{k}}\zeta _{0}^{k} 
\left[ \zeta _{x_{i}^{k}},S_{\delta }\right] \zeta _{x_{i}^{k}}u,\Lambda
^{s}\left( \lambda _{k}\right) _{x_{j}^{k}}\zeta _{0}^{k}S_{\delta }\zeta
_{0}^{k}u\right)  \notag \\
&&+\left( \left[ \zeta _{x_{i}^{k}},\Lambda ^{s}\right] \left( \lambda
_{k}\right) _{x_{j}^{k}}\zeta _{0}^{k}S_{\delta }\zeta _{x_{i}^{k}}u,\Lambda
^{s}\left( \lambda _{k}\right) _{x_{j}^{k}}\zeta _{0}^{k}S_{\delta }\zeta
_{0}^{k}u\right)  \notag \\
&&+\left( \Lambda ^{s}\left( \lambda _{k}\right) _{x_{j}^{k}}\zeta
_{0}^{k}S_{\delta }\zeta _{x_{i}^{k}}u,\left[ \Lambda ^{s},\zeta _{x_{i}^{k}}%
\right] \left( \lambda _{k}\right) _{x_{j}^{k}}\zeta _{0}^{k}S_{\delta
}\zeta _{0}^{k}u\right)  \notag \\
&&+\left( \Lambda ^{s}\left( \lambda _{k}\right) _{x_{j}^{k}}\zeta
_{0}^{k}S_{\delta }\zeta _{x_{i}^{k}}u,\Lambda ^{s}\left( \lambda
_{k}\right) _{x_{j}^{k}}\zeta _{0}^{k}\left[ S_{\delta },\zeta _{x_{i}^{k}}%
\right] \zeta _{0}^{k}u\right)  \notag \\
&\leq &\left\Vert \zeta _{0}^{k}S_{\delta }\left( \zeta _{x_{i}^{k}}\right)
^{2}u\right\Vert _{s}^{2}+\left\Vert \left( \left( \lambda _{k}\right)
_{x_{j}^{k}}\right) ^{2}\zeta _{0}^{k}S_{\delta }\zeta _{0}^{k}u\right\Vert
_{s}^{2}+C\left\Vert \tilde{\zeta}S_{\delta }\tilde{\zeta}u\right\Vert
_{s-1/2}^{2}.  \label{VI-2}
\end{eqnarray}%
The first term on the right is bounded by%
\begin{eqnarray*}
\left\Vert \zeta _{0}^{k}S_{\delta }\left( \zeta _{x_{i}^{k}}\right)
^{2}u\right\Vert _{s}^{2} &\leq &\left\Vert \left( \zeta _{x_{i}^{k}}\right)
^{2}\Lambda ^{s}\zeta _{0}^{k}S_{\delta }\tilde{\zeta}u\right\Vert
^{2}+\left\Vert \left[ \Lambda ^{s},\left( \zeta _{x_{i}^{k}}\right) ^{2}%
\right] \zeta _{0}^{k}S_{\delta }\tilde{\zeta}u\right\Vert ^{2} \\
&&+\left\Vert \Lambda ^{s}\zeta _{0}^{k}\left[ S_{\delta },\left( \zeta
_{x_{i}^{k}}\right) ^{2}\right] \tilde{\zeta}u\right\Vert ^{2}.
\end{eqnarray*}%
By the Wirtinger inequality (\ref{ineq-Wirtinger}), it follows that%
\begin{eqnarray}
\left\Vert \zeta _{0}^{k}S_{\delta }\left( \zeta _{x_{i}^{k}}\right)
^{2}u\right\Vert _{s}^{2} &\leq &C\left\{ \left\Vert \zeta \Lambda ^{s}\zeta
_{0}^{k}S_{\delta }\tilde{\zeta}u\right\Vert ^{2}+\left\Vert \tilde{\zeta}%
S_{\delta }\tilde{\zeta}u\right\Vert _{s-1}^{2}\right\}  \notag \\
&\leq &C\left\{ \left\Vert \tilde{\zeta}S_{\delta }\zeta u\right\Vert
_{s}^{2}+\left\Vert \tilde{\zeta}S_{\delta }\tilde{\zeta}u\right\Vert
_{s-1}^{2}\right\} .  \label{VI-3}
\end{eqnarray}%
Similarly, the second term on the right of (\ref{VI-2}) is bounded by%
\begin{equation*}
\left\Vert \left( \left( \lambda _{k}\right) _{x_{j}^{k}}\right) ^{2}\zeta
_{0}^{k}S_{\delta }\zeta _{0}^{k}u\right\Vert _{s}^{2}\leq C\left\{
\left\Vert \lambda _{k}\zeta _{0}^{k}S_{\delta }\zeta _{0}^{k}u\right\Vert
_{s}^{2}+\left\Vert \tilde{\zeta}S_{\delta }\tilde{\zeta}u\right\Vert
_{s-1}^{2}\right\} .
\end{equation*}%
We obtain%
\begin{equation*}
\left\Vert \left( \lambda _{k}\right) _{x_{j}^{k}}\zeta _{0}^{k}S_{\delta
}\zeta _{x_{i}^{k}}u\right\Vert _{s}^{2}\leq C\left\{ \left\Vert \tilde{\zeta%
}S_{\delta }\zeta u\right\Vert _{s}^{2}+\left\Vert \lambda _{k}\zeta
_{0}^{k}S_{\delta }\zeta _{0}^{k}u\right\Vert _{s}^{2}+\left\Vert \tilde{%
\zeta}S_{\delta }\tilde{\zeta}u\right\Vert _{s-1/2}^{2}\right\} .
\end{equation*}%
Plugging this estimate on the right of (\ref{VI-1}) yields%
\begin{eqnarray}
\left\vert VI_{1}^{kij}\right\vert &\leq &C\left\{ \left\Vert \tilde{\zeta}%
S_{\delta }\zeta u\right\Vert _{s}^{2}+\left\Vert \zeta ^{\prime }S_{\delta
}\zeta u\right\Vert _{s-1}^{2}+\left\Vert \lambda _{k}\zeta
_{0}^{k}S_{\delta }\zeta _{0}^{k}u\right\Vert _{s}^{2}\right\}  \label{VI-a}
\\
&&+C\left\Vert \tilde{\zeta}S_{\delta }\tilde{\zeta}u\right\Vert
_{s-1/2}^{2}.  \notag
\end{eqnarray}

It easily follows that%
\begin{eqnarray}
\left\vert VI_{2}^{kij}\right\vert &\leq &2\left\vert \left( \tilde{\zeta}%
\Lambda ^{s}S_{\delta }\zeta u,\tilde{\zeta}_{x_{j}^{k}}\Lambda
^{s}S_{\delta }\lambda _{k}a_{ij}^{k}\zeta _{x_{i}^{k}}u\right) \right\vert 
\notag \\
&\leq &\left\{ \left\Vert \tilde{\zeta}S_{\delta }\zeta u\right\Vert
_{s}^{2}+\left\Vert \lambda _{k}\zeta _{0}^{k}S_{\delta }\zeta
_{0}^{k}u\right\Vert _{s}^{2}+\left\Vert \zeta ^{\prime }S_{\delta }\tilde{%
\zeta}u\right\Vert _{s-1}^{2}\right\} .  \label{VI-b}
\end{eqnarray}

Commuting $\zeta _{x_{i}^{k}}$ into the left, we have that 
\begin{eqnarray*}
\left\vert VI_{3}^{kij}\right\vert &\leq &\left\vert \left( \tilde{\zeta}%
\Lambda ^{s}S_{\delta }\zeta _{x_{j}^{k}}\zeta _{x_{i}^{k}}u,\tilde{\zeta}%
\Lambda ^{s}S_{\delta }\lambda _{k}a_{ij}^{k}\left( \zeta _{0}^{k}\right)
^{2}u\right) \right\vert \\
&&+\left\vert \left( \tilde{\zeta}\left[ \zeta _{x_{i}^{k}},\Lambda
^{s}S_{\delta }\right] \zeta _{x_{j}^{k}}u,\tilde{\zeta}\Lambda
^{s}S_{\delta }\lambda _{k}a_{ij}^{k}\left( \zeta _{0}^{k}\right)
^{2}u\right) \right\vert \\
&&+\left\vert \left( \tilde{\zeta}\Lambda ^{s}S_{\delta }\zeta
_{x_{j}^{k}}\left( \zeta _{0}^{k}\right) ^{2}u,\tilde{\zeta}\left[ \Lambda
^{s}S_{\delta },\zeta _{x_{i}^{k}}\right] \lambda _{k}a_{ij}^{k}u\right)
\right\vert \\
&\leq &C\left\{ \left\Vert \tilde{\zeta}S_{\delta }\left( \zeta
_{x_{j}^{k}}\right) ^{2}u\right\Vert _{s}^{2}+\left\Vert \lambda _{k}\zeta
_{0}^{k}S_{\delta }\zeta _{0}^{k}u\right\Vert _{s}^{2}+\left\Vert \tilde{%
\zeta}S_{\delta }\tilde{\zeta}u\right\Vert _{s-1}^{2}\right\} .
\end{eqnarray*}%
Applying (\ref{VI-3}) to the first term on the right we obtain%
\begin{equation}
\left\vert VI_{3}^{kij}\right\vert \leq C\left\{ \left\Vert \tilde{\zeta}%
S_{\delta }\zeta u\right\Vert _{s}^{2}+\left\Vert \lambda _{k}\zeta
_{0}^{k}S_{\delta }\zeta _{0}^{k}u\right\Vert _{s}^{2}+\left\Vert \zeta
^{\prime }S_{\delta }\tilde{\zeta}u\right\Vert _{s-1}^{2}\right\} .
\label{VI-c}
\end{equation}

Using the identity (\ref{comm_lemm1}) for $\left[ \zeta ,\Lambda
^{s}S_{\delta }\right] $, we obtain 
\begin{eqnarray}
\left\vert VI_{4}^{kij}\right\vert &\leq &\sum_{r=1}^{m}\sum_{\ell
=1}^{n_{k}}\left\vert \left( \tilde{\zeta}R_{s,\delta }^{0}\Lambda
^{s-2}\zeta _{x_{\ell }^{r}}\frac{\partial }{\partial x_{\ell }^{r}}%
S_{\delta }\tilde{\zeta}\dfrac{\partial }{\partial x_{j}^{k}}u,\tilde{\zeta}%
\Lambda ^{s}S_{\delta }\lambda _{k}a_{ij}^{k}\zeta _{x_{i}^{k}}u\right)
\right\vert  \notag \\
&&+\left\vert \left( \tilde{\zeta}Q_{s,\delta }^{-0}\Lambda ^{s-2}\zeta
_{0}S_{\delta }\tilde{\zeta}\dfrac{\partial }{\partial x_{j}^{k}}u,\tilde{%
\zeta}\Lambda ^{s}S_{\delta }\lambda _{k}a_{ij}^{k}\zeta
_{x_{i}^{k}}u\right) \right\vert  \notag \\
&\leq &\sum_{r=1}^{m}\sum_{\ell =1}^{n_{k}}\left\vert \left( \tilde{\zeta}%
R_{s,\delta }^{0}\Lambda ^{s-2}\zeta _{x_{\ell }^{r}}\frac{\partial }{%
\partial x_{\ell }^{r}}S_{\delta }\tilde{\zeta}\dfrac{\partial }{\partial
x_{j}^{k}}u,\tilde{\zeta}\Lambda ^{s}S_{\delta }\lambda _{k}a_{ij}^{k}\zeta
_{x_{i}^{k}}u\right) \right\vert  \label{VI-4} \\
&&+C\left\{ \left\Vert \lambda _{k}\zeta _{0}^{k}S_{\delta }\zeta
_{0}^{k}u\right\Vert _{s}^{2}+\left\Vert \zeta ^{\prime }S_{\delta }\zeta
^{\prime }u\right\Vert _{s-1}^{2}\right\} .  \notag
\end{eqnarray}%
We treat the first term on the right of (\ref{VI-4}) in a similar way as we
obtain (\ref{VI-2}), we commute $\zeta _{x_{i}^{k}}$ into the left. We
proceed in the same way with $VI_{5}^{kij}$, to obtain%
\begin{equation}
\left\vert VI_{4}^{kij}\right\vert +\left\vert VI_{5}^{kij}\right\vert \leq
C\left\{ \left\Vert \tilde{\zeta}S_{\delta }\zeta u\right\Vert
_{s}^{2}+\left\Vert \lambda _{k}\zeta _{0}^{k}S_{\delta }\zeta
_{0}^{k}u\right\Vert _{s}^{2}+\left\Vert \zeta ^{\prime }S_{\delta }\zeta
^{\prime }u\right\Vert _{s-1/2}^{2}\right\} .  \label{VI-de}
\end{equation}

The treatment of $VI_{6}^{kij}$ and $VI_{7}^{kij}$ is similar. Applying (\ref%
{comm_lemm1}) to $\left[ \lambda _{k}a_{ij}^{k}\zeta _{x_{i}^{k}},\Lambda
^{s}S_{\delta }\right] $, we obtain%
\begin{eqnarray}
\left\vert VI_{6}^{kij}\right\vert &\leq &\sum_{r=1}^{m}\sum_{\ell
=1}^{n_{k}}\left\vert \left( \tilde{\zeta}\Lambda ^{s}S_{\delta }\tilde{\zeta%
}\dfrac{\partial }{\partial x_{j}^{k}}u,\tilde{\zeta}R_{s,\delta
}^{0}\Lambda ^{s-2}\left( \lambda _{k}a_{ij}^{k}\zeta _{x_{i}^{k}}\right)
_{x_{\ell }^{r}}\frac{\partial }{\partial x_{\ell }^{r}}S_{\delta }\zeta
u\right) \right\vert  \notag \\
&&+\left\vert \left( \tilde{\zeta}\Lambda ^{s}S_{\delta }\tilde{\zeta}\dfrac{%
\partial }{\partial x_{j}^{k}}u,\tilde{\zeta}Q_{s,\delta }^{-0}\Lambda
^{s-2}\zeta _{0}S_{\delta }\zeta u\right) \right\vert  \notag \\
&\leq &\sum_{r=1}^{m}\sum_{\ell =1}^{n_{k}}\left\vert \left( \tilde{\zeta}%
\Lambda ^{s}S_{\delta }\tilde{\zeta}\dfrac{\partial }{\partial x_{j}^{k}}u,%
\tilde{\zeta}R_{s,\delta }^{0}\Lambda ^{s-2}\left( \lambda
_{k}a_{ij}^{k}\zeta _{x_{i}^{k}}\right) _{x_{\ell }^{r}}\frac{\partial }{%
\partial x_{\ell }^{r}}S_{\delta }\zeta u\right) \right\vert  \label{VI-5} \\
&&+C\left\{ \left\Vert \tilde{\zeta}S_{\delta }\zeta u\right\Vert
_{s}^{2}+\left\Vert \zeta ^{\prime }S_{\delta }\zeta ^{\prime }u\right\Vert
_{s-1}^{2}\right\} .  \notag
\end{eqnarray}%
We write the terms in the sum above as%
\begin{eqnarray*}
&&\left( \tilde{\zeta}\Lambda ^{s}S_{\delta }\tilde{\zeta}\dfrac{\partial }{%
\partial x_{j}^{k}}u,\tilde{\zeta}R_{s,\delta }^{0}\Lambda ^{s-2}\left(
\lambda _{k}a_{ij}^{k}\zeta _{x_{i}^{k}}\right) _{x_{\ell }^{r}}\frac{%
\partial }{\partial x_{\ell }^{r}}S_{\delta }\zeta u\right) \\
&=&-\left( \frac{\partial }{\partial x_{\ell }^{r}}\left( \lambda
_{k}a_{ij}^{k}\zeta _{x_{i}^{k}}\right) _{x_{\ell }^{r}}\Lambda
^{-2}R_{s,\delta }^{0}\tilde{\zeta}\Lambda ^{s}S_{\delta }\tilde{\zeta}%
\dfrac{\partial }{\partial x_{j}^{k}}u,\Lambda ^{s}\tilde{\zeta}S_{\delta
}\zeta u\right) \\
&&+\left( \tilde{\zeta}\Lambda ^{s}S_{\delta }\tilde{\zeta}\dfrac{\partial }{%
\partial x_{j}^{k}}u,R_{s,\delta }^{0}\left[ \Lambda ^{s}\tilde{\zeta}%
,\Lambda ^{-2}\left( \lambda _{k}a_{ij}^{k}\zeta _{x_{i}^{k}}\right)
_{x_{\ell }^{r}}\frac{\partial }{\partial x_{\ell }^{r}}\right] S_{\delta
}\zeta u\right) \\
&&+\left( \tilde{\zeta}\Lambda ^{s}S_{\delta }\tilde{\zeta}\dfrac{\partial }{%
\partial x_{j}^{k}}u,\left[ \tilde{\zeta},R_{s,\delta }^{0}\Lambda ^{s-2}%
\right] \Lambda ^{-2}\left( \lambda _{k}a_{ij}^{k}\zeta _{x_{i}^{k}}\right)
_{x_{\ell }^{r}}\frac{\partial }{\partial x_{\ell }^{r}}S_{\delta }\zeta
u\right) .
\end{eqnarray*}%
Plugging this into (\ref{VI-5}), we obtain%
\begin{eqnarray*}
\left\vert VI_{6}^{kij}\right\vert &\leq &\sum_{r=1}^{m}\sum_{\ell
=1}^{n_{k}}\left\vert \left( \frac{\partial }{\partial x_{\ell }^{r}}\left(
\lambda _{k}a_{ij}^{k}\zeta _{x_{i}^{k}}\right) _{x_{\ell }^{r}}\Lambda
^{-2}R_{s,\delta }^{0}\tilde{\zeta}\Lambda ^{s}S_{\delta }\tilde{\zeta}%
\dfrac{\partial }{\partial x_{j}^{k}}u,\Lambda ^{s}\tilde{\zeta}S_{\delta
}\zeta u\right) \right\vert \\
&&+C\left\{ \left\Vert \tilde{\zeta}S_{\delta }\zeta u\right\Vert
_{s}^{2}+\left\Vert \zeta ^{\prime }S_{\delta }\zeta ^{\prime }u\right\Vert
_{s-1/2}^{2}\right\} .
\end{eqnarray*}%
We estimate the first term on the right in the same way as we did (\ref{VI-1}%
-\ref{VI-2}). Since $VI_{7}^{kij}$ can be estimated by the same procedure,
we obtain%
\begin{equation}
\left\vert VI_{6}^{kij}\right\vert +\left\vert VI_{7}^{kij}\right\vert \leq
C\left\{ \left\Vert \tilde{\zeta}S_{\delta }\zeta u\right\Vert
_{s}^{2}+\left\Vert \lambda _{k}\zeta _{0}^{k}S_{\delta }\zeta
_{0}^{k}u\right\Vert _{s}^{2}+\left\Vert \zeta ^{\prime }S_{\delta }\zeta
^{\prime }u\right\Vert _{s-1/2}^{2}\right\} .  \label{VI-fg}
\end{equation}%
Combining estimates (\ref{VI-a}), (\ref{VI-b}), (\ref{VI-c}), (\ref{VI-de}),
and (\ref{VI-fg}) yields%
\begin{eqnarray}
\left\vert VI\right\vert &\leq
&2\sum\limits_{k=1}^{m}\sum\limits_{i,j=1}^{n_{k}}\left\vert
VI^{kij}\right\vert  \notag \\
&\leq &C\left\{ \left\Vert \tilde{\zeta}S_{\delta }\zeta u\right\Vert
_{s}^{2}+\sum\limits_{k=1}^{m}\left\Vert \lambda _{k}\zeta _{0}^{k}S_{\delta
}\zeta _{0}^{k}u\right\Vert _{s}^{2}+\left\Vert \zeta ^{\prime }S_{\delta
}\zeta ^{\prime }u\right\Vert _{s-1/2}^{2}\right\} .  \label{comm_lemm-VI}
\end{eqnarray}

Applying (\ref{comm_lemm-I}), (\ref{comm_lemm-II-III}), (\ref{comm_lemm-IV-V}%
), and (\ref{comm_lemm-VI}) to the right of (\ref{comm_lemm-00}) we obtain
the conclusion of the lemma.
\end{proof}

\begin{corollary}
\label{cor_hard} Under the assumptions of Lemma \ref{hard} the following
estimate holds%
\begin{equation*}
\begin{array}{lll}
\sum\limits_{k=1}^{m}\left\Vert \sqrt{\lambda _{k}}\Lambda
_{x^{k}}^{\varepsilon }\tilde{\zeta}\Lambda ^{s}S_{\delta }\zeta \mathbf{u}%
\right\Vert ^{2} & \leq & C\left\{ \left\Vert \tilde{\zeta}S_{\delta }\zeta 
\mathbf{u}\right\Vert _{s}^{2}+\left\vert \left( \tilde{\zeta}\Lambda
^{s}S_{\delta }\tilde{\zeta}\mathbf{Lu},\tilde{\zeta}\Lambda ^{s}S_{\delta
}\zeta \mathbf{u}\right) \right\vert \right. \\ 
&  & \left. \quad +\sum\limits_{k=1}^{m}\left\Vert \lambda _{k}\zeta
_{0}^{k}S_{\delta }\zeta _{0}^{k}\mathbf{u}\right\Vert _{s}^{2}+\left\Vert
\zeta ^{\prime }S_{\delta }\zeta ^{\prime }\mathbf{u}\right\Vert
_{s-1/2}^{2}\right\} .%
\end{array}%
\end{equation*}
for all $\mathbf{u}\in \prod_{k=1}^{N}H^{s-1}(\mathbb{R}^{n})$.
\end{corollary}

\begin{proof}
Since $\mathbf{u}\in \prod_{k=1}^{N}H^{s-1}(\mathbb{R}^{n})$ we have that $%
\Lambda ^{s}S_{\delta }\zeta \mathbf{u}\in \prod_{k=1}^{N}H^{2}(\mathbb{R}%
^{n})$ so we can replace $\zeta $ by $\tilde{\zeta}$ and $\mathbf{u}$ by $%
\Lambda ^{s}S_{\delta }\zeta \mathbf{u}$ in Lemma \ref{inter_lem} to obtain%
\begin{equation*}
\sum\limits_{k=1}^{m}\left\Vert \sqrt{\lambda _{k}}\Lambda
_{x^{k}}^{\varepsilon }\left( \tilde{\zeta}\Lambda ^{s}S_{\delta }\zeta 
\mathbf{u}\right) \right\Vert ^{2}\leq C\left\{ \left\vert \left( \tilde{%
\zeta}\mathbf{L}\Lambda ^{s}S_{\delta }\zeta \mathbf{u},\tilde{\zeta}\Lambda
^{s}S_{\delta }\zeta \mathbf{u}\right) \right\vert +\left\Vert \zeta
^{\prime }\Lambda ^{s}S_{\delta }\zeta \mathbf{u}\right\Vert ^{2}\right\} .
\end{equation*}%
The corollary follows by commuting $\mathbf{L}$ with $\Lambda ^{s}S_{\delta
}\zeta $ and applying Lemma \ref{hard}.
\end{proof}

\begin{lemma}
\label{rem_term2}There exists $C>0$ such that%
\begin{equation*}
\sum\limits_{k=1}^{m}\left\Vert \lambda _{k}\zeta _{0}^{k}S_{\delta }\zeta
_{0}^{k}\mathbf{u}\right\Vert _{s}^{2}\leq C\left\{ \left\vert \left( \zeta
^{\prime }\Lambda ^{s-\varepsilon }S_{\delta }\zeta ^{\prime }\mathbf{Lu}%
,\zeta ^{\prime }\Lambda ^{s-\varepsilon }S_{\delta }\tilde{\zeta}\mathbf{u}%
\right) \right\vert +\left\Vert \zeta ^{\prime \prime }S_{\delta }\zeta
^{\prime \prime }\mathbf{u}\right\Vert _{s-\varepsilon }^{2}\right\}
\end{equation*}
\end{lemma}

\begin{proof}
We have%
\begin{eqnarray}
\sum\limits_{k=1}^{m}\left\Vert \lambda _{k}\zeta _{0}^{k}S_{\delta }\zeta
_{0}^{k}\mathbf{u}\right\Vert _{s}^{2}
&=&\sum\limits_{k=1}^{m}\sum\limits_{p=1}^{N}\left\Vert \Lambda ^{s}\lambda
_{k}\zeta _{0}^{k}S_{\delta }\zeta _{0}^{k}u_{p}\right\Vert ^{2}  \notag \\
&\leq &\sum\limits_{k=1}^{m}\sum\limits_{p=1}^{N}\left\Vert \Lambda
^{\varepsilon }\lambda _{k}\zeta _{0}^{k}\Lambda ^{s-\varepsilon }S_{\delta }%
\tilde{\zeta}u_{p}\right\Vert ^{2}+C\left\Vert \zeta ^{\prime }S_{\delta }%
\tilde{\zeta}\mathbf{u}\right\Vert _{s-1}^{2}  \notag \\
&\leq &C\sum\limits_{k,\ell =1}^{m}\sum\limits_{p=1}^{N}\left\Vert \Lambda
_{x^{\ell }}^{\varepsilon }\lambda _{k}\zeta _{0}^{k}\Lambda ^{s-\varepsilon
}S_{\delta }\tilde{\zeta}u_{p}\right\Vert ^{2}+C\left\Vert \zeta ^{\prime
}S_{\delta }\tilde{\zeta}\mathbf{u}\right\Vert _{s-1}^{2}  \notag \\
&\leq &C\sum\limits_{k,\ell =1}^{m}\sum\limits_{p=1}^{N}\left\Vert \lambda
_{k}\zeta _{0}^{k}\Lambda _{x^{\ell }}^{\varepsilon }\tilde{\zeta}\Lambda
^{s-\varepsilon }S_{\delta }\tilde{\zeta}u_{p}\right\Vert ^{2}+C\left\Vert
\zeta ^{\prime }S_{\delta }\tilde{\zeta}\mathbf{u}\right\Vert _{s-1}^{2}.
\label{term2-1}
\end{eqnarray}

Now we recall that by the hypotheses on $\lambda _{k}$, we have that if $%
1\leq \ell \leq m$, and $\ell \neq k$ then $\lambda _{\ell }\left( x\right)
>0$ for all $x$ in an open neighbourhood of the support of $\zeta _{0}^{k}$.
Hence%
\begin{equation*}
\delta _{k}=\min_{\substack{ 1\leq \ell \leq m  \\ \ell \neq k}}\inf_{x\in 
\mathrm
{support}\left( \zeta _{0}^{k}\right) }\lambda _{\ell }\left(
x\right) >0,\qquad 1\leq \ell \leq m.
\end{equation*}%
Consequently, 
\begin{equation}
\lambda _{k}\zeta _{0}^{k}\leq \frac{\left\Vert \lambda _{k}\right\Vert
_{\infty }}{\delta _{k}}\lambda _{\ell }\left( x\right) \zeta
_{0}^{k},\qquad 1\leq \ell ,k\leq m.  \label{ineq-lambdas}
\end{equation}

We apply (\ref{ineq-lambdas}) to the right side of (\ref{term2-1}) to obtain%
\begin{eqnarray*}
\sum\limits_{k=1}^{m}\left\Vert \lambda _{k}\zeta _{0}^{k}S_{\delta }\zeta
_{0}^{k}\mathbf{u}\right\Vert _{s}^{2} &\leq
&C\sum\limits_{k=1}^{m}\sum\limits_{p=1}^{N}\left\Vert \lambda _{k}\zeta
_{0}^{k}\Lambda _{x^{k}}^{\varepsilon }\tilde{\zeta}\Lambda ^{s-\varepsilon
}S_{\delta }\tilde{\zeta}u_{p}\right\Vert ^{2}+C\left\Vert \zeta ^{\prime
}S_{\delta }\tilde{\zeta}\mathbf{u}\right\Vert _{s-1}^{2} \\
&\leq &C\sum\limits_{k=1}^{m}\left\Vert \sqrt{\lambda _{k}}\Lambda
_{x^{k}}^{\varepsilon }\tilde{\zeta}\Lambda ^{s-\varepsilon }S_{\delta }%
\tilde{\zeta}\mathbf{u}\right\Vert ^{2}+C\left\Vert \zeta ^{\prime
}S_{\delta }\tilde{\zeta}\mathbf{u}\right\Vert _{s-1}^{2}.
\end{eqnarray*}%
We apply Corollary \ref{cor_hard} with $\tilde{\zeta}$ instead of $\zeta $, $%
\zeta ^{\prime }$ instead of $\tilde{\zeta}$, etc., to the first term on the
right to obtain%
\begin{eqnarray*}
\sum\limits_{k=1}^{m}\left\Vert \lambda _{k}\zeta _{0}^{k}S_{\delta }\zeta
_{0}^{k}\mathbf{u}\right\Vert _{s}^{2} &\leq
&C\sum\limits_{k=1}^{m}\left\Vert \sqrt{\lambda _{k}}\Lambda
_{x^{k}}^{\varepsilon }\tilde{\zeta}\Lambda ^{s-\varepsilon }S_{\delta }%
\tilde{\zeta}\mathbf{u}\right\Vert ^{2}+C\left\Vert \zeta ^{\prime
}S_{\delta }\tilde{\zeta}\mathbf{u}\right\Vert _{s-1}^{2} \\
&\leq &C\left\{ \left\Vert \zeta ^{\prime }S_{\delta }\tilde{\zeta}\mathbf{u}%
\right\Vert _{s-\varepsilon }^{2}+\left\vert \left( \zeta ^{\prime }\Lambda
^{s-\varepsilon }S_{\delta }\zeta ^{\prime }\mathbf{Lu},\zeta ^{\prime
}\Lambda ^{s-\varepsilon }S_{\delta }\tilde{\zeta}\mathbf{u}\right)
\right\vert \quad \right. \\
&&\left. +\sum\limits_{k=1}^{m}\left\Vert \lambda _{k}\tilde{\zeta}%
_{0}^{k}S_{\delta }\tilde{\zeta}_{0}^{k}\mathbf{u}\right\Vert
_{s-\varepsilon }^{2}+\left\Vert \zeta ^{\prime \prime }S_{\delta }\zeta
^{\prime \prime }\mathbf{u}\right\Vert _{s-\varepsilon -1/2}^{2}+C\left\Vert
\zeta ^{\prime }S_{\delta }\tilde{\zeta}\mathbf{u}\right\Vert
_{s-1}^{2}\right\} . \\
&\leq &C\left\{ \left\vert \left( \zeta ^{\prime }\Lambda ^{s-\varepsilon
}S_{\delta }\zeta ^{\prime }\mathbf{Lu},\zeta ^{\prime }\Lambda
^{s-\varepsilon }S_{\delta }\tilde{\zeta}\mathbf{u}\right) \right\vert
+\left\Vert \zeta ^{\prime \prime }S_{\delta }\zeta ^{\prime \prime }\mathbf{%
u}\right\Vert _{s-\varepsilon }^{2}\right\} .
\end{eqnarray*}
\end{proof}

The last auxiliary result we need is the following Poincar\'{e}-type
inequality.

\begin{lemma}[Poincar\'{e}-type inequality]
\label{lemma-Poinc}For every $0<\varepsilon <1$ and cutoff $\sigma _{1}$
such that $d_{0}=\mathrm
{diam}\left( \mathrm
{support}\,\sigma _{1}\right)
\leq \sqrt{\varepsilon /2}$, there exists $C>0$ such that for any $\mathbf{u}%
\in \prod_{k=1}^{N}H^{s+1}\left( \mathbb{R}^{n}\right) $%
\begin{equation*}
\left\Vert \tilde{\zeta}\Lambda ^{s}\zeta \mathbf{u}\right\Vert \leq
Cd_{0}^{\alpha \left( \varepsilon \right) }\left\{ \left\Vert \Lambda
_{x^{1}}^{\varepsilon }\tilde{\zeta}\Lambda ^{s}\zeta \mathbf{u}\right\Vert
+d_{0}^{-1}\left\Vert \zeta \mathbf{u}\right\Vert _{s-1}\right\} ,
\end{equation*}%
where $\alpha \left( \varepsilon \right) =1/\left( 4\log _{2}\left( \dfrac{1%
}{\varepsilon }+1\right) \right) $.
\end{lemma}

\begin{proof}
We can write 
\begin{equation*}
\left\Vert \tilde{\zeta}\Lambda ^{s}\zeta \mathbf{u}\right\Vert
^{2}=\sum\limits_{p=1}^{N}||\tilde{\zeta}\Lambda ^{s}\zeta
u_{p}||^{2}=\sum\limits_{p=1}^{N}\left( \Lambda _{x^{1}}^{\varepsilon }%
\tilde{\zeta}\Lambda ^{s}\zeta u_{p},\Lambda _{x^{1}}^{-\varepsilon }\tilde{%
\zeta}\Lambda ^{s}\zeta u_{p}\right)
\end{equation*}%
For simplicity, we write $u=u_{p}$. Using Young's inequality, we obtain 
\begin{equation*}
\left\Vert \tilde{\zeta}\Lambda ^{s}\zeta u\right\Vert ^{2}\leq a\left\Vert
\Lambda _{x^{1}}^{\varepsilon }\tilde{\zeta}\Lambda ^{s}\zeta u\right\Vert
^{2}+\frac{1}{4a}\left\Vert \Lambda _{x^{1}}^{-\varepsilon }\tilde{\zeta}%
\Lambda ^{s}\zeta u\right\Vert ^{2}
\end{equation*}%
For the second term on the right we similarly obtain 
\begin{equation*}
\left\Vert \Lambda _{x^{1}}^{-\varepsilon }\tilde{\zeta}\Lambda ^{s}\zeta
u\right\Vert ^{2}\leq 4a^{2}\left\Vert \Lambda _{x^{1}}^{\varepsilon }\tilde{%
\zeta}\Lambda ^{s}\zeta u\right\Vert ^{2}+\frac{1}{4^{2}a^{2}}\left\Vert
\Lambda _{x^{1}}^{-3\varepsilon }\tilde{\zeta}\Lambda ^{s}\zeta u\right\Vert
^{2}.
\end{equation*}%
Iterating $N$ times yields%
\begin{equation*}
\left\Vert \tilde{\zeta}\Lambda ^{s}\zeta u\right\Vert ^{2}\leq \frac{2a}{%
1-4^{2}a^{2}}\left\Vert \Lambda _{x^{1}}^{\varepsilon }\tilde{\zeta}\Lambda
^{s}\zeta u\right\Vert ^{2}+\frac{1}{\left( 4^{2}a^{2}\right) ^{N-1/2}}%
\left\Vert \Lambda _{x^{1}}^{-\left( 2^{N}-1\right) \varepsilon }\tilde{\zeta%
}\Lambda ^{s}\zeta u\right\Vert ^{2}.
\end{equation*}%
We choose an integer $N$ such that $\left( 2^{N}-1\right) \varepsilon \geq
1>\left( 2^{N-1}-1\right) \varepsilon $, that is,%
\begin{equation}
N-1<\log _{2}\left( \dfrac{1}{\varepsilon }+1\right) \leq N.
\label{ineq-P-0}
\end{equation}%
Using Poincar\'{e}'s inequality with respect to $x^{1}$, and the
monotonicity of the $||\cdot ||_{s}$ norms we have%
\begin{eqnarray*}
\left\Vert \Lambda _{x^{1}}^{-\left( 2^{N}-1\right) \varepsilon }\tilde{\zeta%
}\Lambda ^{s}\zeta u\right\Vert ^{2} &\leq &\left\Vert \Lambda
_{x^{1}}^{-1}\Lambda _{x^{1}}^{\varepsilon }\tilde{\zeta}\Lambda ^{s}\zeta
u\right\Vert ^{2}\leq \left\Vert \zeta ^{\prime }\Lambda
_{x^{1}}^{-1}\Lambda _{x^{1}}^{\varepsilon }\tilde{\zeta}\Lambda ^{s}\zeta
u\right\Vert ^{2}+C\left\Vert \zeta u\right\Vert _{s-1}^{2} \\
&\leq &d_{0}^{2}\left\Vert \nabla _{x^{1}}\zeta ^{\prime }\Lambda
_{x^{1}}^{-1}\Lambda _{x^{1}}^{\varepsilon }\tilde{\zeta}\Lambda ^{s}\zeta
u\right\Vert ^{2}+C\left\Vert \zeta u\right\Vert _{s-1}^{2} \\
&\leq &d_{0}^{2}\left\Vert \Lambda _{x^{1}}^{\varepsilon }\tilde{\zeta}%
\Lambda ^{s}\zeta u\right\Vert ^{2}+C\left\Vert \zeta u\right\Vert
_{s-1}^{2}.
\end{eqnarray*}%
Therefore,%
\begin{equation*}
\left\Vert \tilde{\zeta}\Lambda ^{s}\zeta u\right\Vert ^{2}\leq \left( 
\dfrac{2a}{1-4^{2}a^{2}}+\frac{d_{0}^{2}}{\left( 4^{2}a^{2}\right) ^{N-1/2}}%
\right) \left\Vert \Lambda _{x^{1}}^{\varepsilon }\tilde{\zeta}\Lambda
^{s}\zeta u\right\Vert ^{2}+\frac{C}{\left( 4^{2}a^{2}\right) ^{N-1/2}}%
\left\Vert \zeta u\right\Vert _{s-1}^{2}.
\end{equation*}%
We take $4^{2}a^{2}=d_{0}^{2/N}$, note that since $d_{0}^{2}\leq \varepsilon
/2$, and $2/N>1/\log _{2}\left( \dfrac{1}{\varepsilon }+1\right) $, we have%
\begin{equation*}
d_{0}^{2/N}\leq \left( \dfrac{\varepsilon }{2}\right) ^{\frac{1}{N}}<\left( 
\dfrac{\varepsilon }{2}\right) ^{\frac{1}{2\log _{2}\left( \frac{1}{%
\varepsilon }+1\right) }}<\left( \dfrac{1}{\frac{1}{\varepsilon }+1}\right)
^{\frac{1}{2\log _{2}\left( \frac{1}{\varepsilon }+1\right) }}=\dfrac{1}{%
\sqrt{2}}.
\end{equation*}%
Thus, we obtain$.$%
\begin{equation*}
\left\Vert \tilde{\zeta}\Lambda ^{s}\zeta u\right\Vert \leq Cd_{0}^{1/\left(
2N\right) }\left( \left\Vert \Lambda _{x^{1}}^{\varepsilon }\tilde{\zeta}%
\Lambda ^{s}\zeta u\right\Vert +d_{0}^{-1}\left\Vert \zeta u\right\Vert
_{s-1}\right) .
\end{equation*}%
Because of (\ref{ineq-P-0}) this implies the result in the lemma with $%
\alpha \left( \varepsilon \right) =1/\left( 4\log _{2}\left( \dfrac{1}{%
\varepsilon }+1\right) \right) $.
\end{proof}

\begin{lemma}
\label{main_apriori}There exist $\varepsilon ,C>0$ such that for all $%
\mathbf{u}\in \prod_{k=1}^{N}H^{s-1}(\mathbb{R}^{n})$ and all $0<\delta \leq
1$%
\begin{equation}
\begin{tabular}{rcl}
$\left\Vert S_{\delta }\zeta \mathbf{u}\right\Vert _{s}^{2}$ & $\leq $ & $%
Cd_{0}^{\alpha \left( \varepsilon \right) }\left\vert \left( \tilde{\zeta}%
\Lambda ^{s}S_{\delta }\tilde{\zeta}\mathbf{Lu},\tilde{\zeta}\Lambda
^{s}S_{\delta }\zeta \mathbf{u}\right) \right\vert $ \\ 
&  & $Cd_{0}^{\alpha \left( \varepsilon \right) }\left\vert \left( \zeta
^{\prime }\Lambda ^{s-\varepsilon }S_{\delta }\zeta ^{\prime }\mathbf{Lu}%
,\zeta ^{\prime }\Lambda ^{s-\varepsilon }S_{\delta }\tilde{\zeta}\mathbf{u}%
\right) \right\vert $ \\ 
&  & $+\left( Cd_{0}^{-2\left( 1-\alpha \left( \varepsilon \right) \right)
}+1\right) \left\Vert S_{\delta }\zeta ^{\prime \prime }\mathbf{u}%
\right\Vert _{s-\varepsilon }^{2},$%
\end{tabular}
\label{prop1}
\end{equation}%
where $d_{0},\alpha \left( \varepsilon \right) $ are as in Lemma \ref%
{lemma-Poinc}. We also have that%
\begin{equation}
\left\Vert S_{\delta }\zeta \mathbf{u}\right\Vert _{s}\leq C\left\{
d_{0}^{2\alpha \left( \varepsilon \right) }\left\Vert S_{\delta }\zeta
^{\prime }\mathbf{Lu}\right\Vert _{s}^{2}+\left( d_{0}^{-2\left( 1-\alpha
\left( \varepsilon \right) \right) }+1\right) \left\Vert S_{\delta }\zeta
^{\prime \prime }\mathbf{u}\right\Vert _{s-\varepsilon }^{2}\right\} .
\label{prop2}
\end{equation}
\end{lemma}

\begin{proof}
We have%
\begin{eqnarray*}
\left\Vert S_{\delta }\zeta \mathbf{u}\right\Vert _{s}^{2} &=&\left\Vert
\Lambda ^{s}S_{\delta }\tilde{\zeta}^{2}\zeta \mathbf{u}\right\Vert ^{2}\leq
\left\Vert \tilde{\zeta}\Lambda ^{s}S_{\delta }\tilde{\zeta}\zeta \mathbf{u}%
\right\Vert ^{2}+C\left\Vert S_{\delta }\zeta \mathbf{u}\right\Vert
_{s-1}^{2} \\
&\leq &\left\Vert \tilde{\zeta}\Lambda ^{s}\zeta S_{\delta }\tilde{\zeta}%
\mathbf{u}\right\Vert ^{2}+C\left\Vert S_{\delta }\tilde{\zeta}\mathbf{u}%
\right\Vert _{s-1}^{2}.
\end{eqnarray*}

We apply the Poincar\'{e} inequality, Lemma \ref{lemma-Poinc}, to the first
term on the right,%
\begin{eqnarray*}
\left\Vert S_{\delta }\zeta \mathbf{u}\right\Vert _{s}^{2} &\leq
&Cd_{0}^{\alpha \left( \varepsilon \right) }\left\Vert \Lambda
_{x^{1}}^{\varepsilon }\tilde{\zeta}\Lambda ^{s}\zeta S_{\delta }\tilde{\zeta%
}\mathbf{u}\right\Vert +C\left( d_{0}^{\alpha \left( \varepsilon \right)
-1}+1\right) \left\Vert S_{\delta }\tilde{\zeta}\mathbf{u}\right\Vert _{s-1}
\\
&\leq &Cd_{0}^{\alpha \left( \varepsilon \right) }\left\Vert \Lambda
_{x^{1}}^{\varepsilon }\tilde{\zeta}\Lambda ^{s}S_{\delta }\zeta \mathbf{u}%
\right\Vert +C\left( d_{0}^{\alpha \left( \varepsilon \right) -1}+1\right)
\left\Vert S_{\delta }\tilde{\zeta}\mathbf{u}\right\Vert _{s-1}
\end{eqnarray*}%
Since $\lambda _{1}\equiv 1$, by Corollary \ref{cor_hard} we then have%
\begin{eqnarray*}
\left\Vert S_{\delta }\zeta \mathbf{u}\right\Vert _{s}^{2} &\leq
&Cd_{0}^{\alpha \left( \varepsilon \right) }\left\{ \left\Vert S_{\delta
}\zeta \mathbf{u}\right\Vert _{s}^{2}+\left\vert \left( \tilde{\zeta}\Lambda
^{s}S_{\delta }\tilde{\zeta}\mathbf{Lu},\tilde{\zeta}\Lambda ^{s}S_{\delta
}\zeta \mathbf{u}\right) \right\vert +\sum\limits_{k=1}^{m}\left\Vert
\lambda _{k}\zeta _{0}^{k}S_{\delta }\zeta _{0}^{k}\mathbf{u}\right\Vert
_{s}^{2}\right\} \\
&&+C\left( d_{0}^{\alpha \left( \varepsilon \right) -1}+1\right) \left\Vert
S_{\delta }\zeta ^{\prime }\mathbf{u}\right\Vert _{s-1}.
\end{eqnarray*}%
Taking $\varepsilon $ small enough (we assume that at least $\varepsilon
\leq 1$), we absorb the first term on the right into the left, and apply
Lemma \ref{rem_term2} to the third term on the right. We get%
\begin{eqnarray*}
\left\Vert S_{\delta }\zeta \mathbf{u}\right\Vert _{s}^{2} &\leq
&Cd_{0}^{\alpha \left( \varepsilon \right) }\left\{ \left\vert \left( \tilde{%
\zeta}\Lambda ^{s}S_{\delta }\tilde{\zeta}\mathbf{Lu},\tilde{\zeta}\Lambda
^{s}S_{\delta }\zeta \mathbf{u}\right) \right\vert +\left\vert \left( \zeta
^{\prime }\Lambda ^{s-\varepsilon }S_{\delta }\zeta ^{\prime }\mathbf{Lu}%
,\zeta ^{\prime }\Lambda ^{s-\varepsilon }S_{\delta }\tilde{\zeta}\mathbf{u}%
\right) \right\vert \right\} \\
&&+C\left( d_{0}^{\alpha \left( \varepsilon \right) -1}+1\right) \left\Vert
S_{\delta }\zeta ^{\prime \prime }\mathbf{u}\right\Vert _{s-\varepsilon }.
\end{eqnarray*}%
This proves (\ref{prop1}). The second inequality, (\ref{prop2}), follows
from (\ref{prop1}) after another absorption into the left.
\end{proof}

\section{Proof of Theorems \protect\ref{main} and \protect\ref{main_s}\label%
{hypoellipticity}}

Suppose $\mathbf{u}$ is a distribution on $\left( \mathbb{R}^{n}\right) ^{N}$
such that $\mathbf{u}\in \prod_{k=1}^{N}H^{s-1}(\mathbb{R}^{n})$. If
moreover $\zeta \mathbf{Lu}\in H^{s}\left( \mathbb{R}^{n}\right) $ for all $%
\zeta \in C_{0}^{\infty }\left( U\right) $, then, by (\ref{prop2}), for all $%
0<\delta \leq 1$ we have%
\begin{equation*}
\left\Vert S_{\delta }\zeta \mathbf{u}\right\Vert _{s}\leq Cd_{0}^{2\alpha
\left( \varepsilon \right) }\left\Vert S_{\delta }\zeta ^{\prime }\mathbf{Lu}%
\right\Vert _{s}^{2}+C\left( d_{0}^{-2\left( 1-\alpha \left( \varepsilon
\right) \right) }+1\right) \left\Vert S_{\delta }\zeta ^{\prime \prime }%
\mathbf{u}\right\Vert _{s-\varepsilon }^{2},
\end{equation*}%
where the cutoff functions $\zeta \prec \zeta ^{\prime }\prec \zeta ^{\prime
\prime }$ are supported in $U$ and the constants $C$, $\alpha $, $d_{0}$ do
not depend on $\delta $. Letting $\delta \rightarrow 0^{+}$, by (\ref{3_S})
in Lemma \ref{properties_S} we obtain%
\begin{equation}
\left\Vert \zeta \mathbf{u}\right\Vert _{s}\leq Cd_{0}^{2\alpha \left(
\varepsilon \right) }\left\Vert \zeta ^{\prime }\mathbf{Lu}\right\Vert
_{s}^{2}+C\left( d_{0}^{-2\left( 1-\alpha \left( \varepsilon \right) \right)
}+1\right) \left\Vert \zeta ^{\prime \prime }\mathbf{u}\right\Vert
_{s-\varepsilon }^{2},  \label{main-pf-01}
\end{equation}%
This inequality suffices to prove the hypoellipticity, without loss of
derivatives, of $\mathbf{L}$. Indeed, since $\zeta ^{\prime \prime }\mathbf{u%
}$ is a compactly supported distribution there exists $s_{0}\in \mathbb{R}$
such that $\left\Vert \varphi \mathbf{u}\right\Vert _{s_{0}}\leq \infty $
for all $\varphi \in C_{0}^{\infty }\left( U\right) $. If $s_{0}\geq
s-\varepsilon $ then (\ref{main-pf-01}) implies that $\zeta \mathbf{u}\in
H^{s}\left( \mathbb{R}^{n}\right) $ for all $\zeta \in C_{0}^{\infty }\left(
U\right) $. On the other hand, if $s_{0}<s-\varepsilon $, let $N$ be the
positive integer such that $N\varepsilon \leq s-s_{0}<\left( N+1\right)
\varepsilon $. Let $\zeta =\zeta ^{0}\prec \zeta ^{1}\prec \zeta ^{2}\prec
\cdots \prec \zeta ^{2N}$ be a sequence of cutoff functions supported in $U$%
. Iterating (\ref{main-pf-01})we obtain%
\begin{equation*}
\left\Vert \zeta ^{2\left( j-1\right) }\mathbf{u}\right\Vert _{s-\left(
j-1\right) \varepsilon }^{2}\leq Cd_{0}^{2\alpha \left( \varepsilon \right)
}\left\Vert \zeta ^{2j-1}\mathbf{Lu}\right\Vert _{s-\left( j-1\right)
\varepsilon }^{2}+C\left( d_{0}^{-2\left( 1-\alpha \left( \varepsilon
\right) \right) }+1\right) \left\Vert \zeta ^{2j}\mathbf{u}\right\Vert
_{s-j\varepsilon }^{2},
\end{equation*}%
$\qquad $\ for $j=1,\dots ,N$. Assembling these estimates yields%
\begin{eqnarray*}
\left\Vert \zeta \mathbf{u}\right\Vert _{s} &\leq &Cd_{0}^{2\alpha \left(
\varepsilon \right) }\sum_{j=1}^{N}\left( d_{0}^{-2\left( 1-\alpha \left(
\varepsilon \right) \right) }+1\right) ^{j-1}\left\Vert \zeta ^{2j-1}\mathbf{%
Lu}\right\Vert _{s-\left( 2j-1\right) \varepsilon }^{2} \\
&&+C\left( d_{0}^{-2\left( 1-\alpha \left( \varepsilon \right) \right)
}+1\right) ^{2N}\left\Vert \zeta ^{2N}\mathbf{u}\right\Vert _{s-N\varepsilon
}^{2}.
\end{eqnarray*}%
By the monotonicity of the $H^{s}$-norms, and since $s_{0}\leq
s-N\varepsilon $, it follows that%
\begin{equation*}
\begin{array}{rcl}
\left\Vert \zeta \mathbf{u}\right\Vert _{s} & \leq  & CNd_{0}^{2\alpha
\left( \varepsilon \right) }\left( d_{0}^{-2\left( 1-\alpha \left(
\varepsilon \right) \right) }+1\right) ^{N-1}\left\Vert \zeta ^{2N-1}\mathbf{%
Lu}\right\Vert _{s-\varepsilon }^{2} \\ 
&  & +C\left( d_{0}^{-2\left( 1-\alpha \left( \varepsilon \right) \right)
}+1\right) ^{2N}\left\Vert \zeta ^{2N}\mathbf{u}\right\Vert
_{s_{0}}^{2}<\infty 
\end{array}%
\end{equation*}%
Hence $\zeta \mathbf{u}\in H^{s}\left( \mathbb{R}^{n}\right) $ for all $%
\zeta \in C_{0}^{\infty }\left( U\right) $ and this finishes the proof. $%
\hfill \square $

\end{document}